\documentclass{amsart}
\usepackage{amsmath, amssymb, amsthm, epsfig}
\usepackage{amssymb, latexsym}
\usepackage{url}

\DeclareMathOperator{\Aut}{\mathrm{Aut}}
\DeclareMathOperator{\CM}{\mathrm{CM}}

\DeclareMathOperator{\norm}{\mathrm{norm}}
\DeclareMathOperator{\Norm}{\mathrm{Norm}}

\DeclareMathOperator{\rank}{\mathrm{rank}}
\DeclareMathOperator{\Reg}{\mathrm{Reg}}

\DeclareMathOperator{\Tor}{\mathrm{Tor}}

\begin{document}
 \bibliographystyle{plain}

 \newtheorem{theorem}{Theorem}[section]
 \newtheorem{lemma}{Lemma}[section]
 \newtheorem{corollary}{Corollary}[section]
 \newtheorem{conjecture}{Conjecture}
 \newtheorem{definition}{Definition}
 
 \newcommand{\mc}{\mathcal}
 \newcommand{\A}{\mc A}
 \newcommand{\B}{\mc B}
 \newcommand{\cc}{\mc C}
 \newcommand{\D}{\mc D}
 \newcommand{\E}{\mc E}
 \newcommand{\F}{\mc F}
 \newcommand{\G}{\mc G}
 \newcommand{\hH}{\mc H}
 \newcommand{\I}{\mc I}
 \newcommand{\J}{\mc J}
 \newcommand{\K}{\mc K}
 \newcommand{\eL}{\mc L}
 \newcommand{\M}{\mc M}
 \newcommand{\eN}{\mc N}
 \newcommand{\pp}{\mc P}
 \newcommand{\qq}{\mc Q}
 \newcommand{\U}{\mc U}
 \newcommand{\V}{\mc V}
 \newcommand{\W}{\mc W}
 \newcommand{\X}{\mc X}
 \newcommand{\Y}{\mc Y}
 \newcommand{\zZ}{\mc Z}
 \newcommand{\C}{\mathbb{C}}
 \newcommand{\R}{\mathbb{R}}
 \newcommand{\Q}{\mathbb{Q}}
 \newcommand{\T}{\mathbb{T}}
 \newcommand{\Z}{\mathbb{Z}}
 \newcommand{\aA}{\mathfrak A}
 \newcommand{\bB}{\mathfrak B}
 \newcommand{\cC}{\mathfrak C}
 \newcommand{\dD}{\mathfrak D}
 \newcommand{\ee}{\mathfrak E}
 \newcommand{\ff}{\mathfrak F}
 \newcommand{\iI}{\mathfrak I}
 \newcommand{\nN}{\mathfrak N}
 \newcommand{\pP}{\mathfrak P}
 \newcommand{\uU}{\mathfrak U}
 \newcommand{\fb}{f_{\beta}}
 \newcommand{\fg}{f_{\gamma}}
 \newcommand{\gb}{g_{\beta}}
 \newcommand{\vphi}{\varphi}
 \newcommand{\p}{\varphi}
 \newcommand{\ep}{\varepsilon}
 \newcommand{\bo}{\boldsymbol 0}
 \newcommand{\bone}{\boldsymbol 1}
 \newcommand{\ba}{\boldsymbol a}
 \newcommand{\bb}{\boldsymbol b}
 \newcommand{\bc}{\boldsymbol c}
 \newcommand{\be}{\boldsymbol e}
 \newcommand{\bff}{\boldsymbol f}
 \newcommand{\bk}{\boldsymbol k}
 \newcommand{\bl}{\boldsymbol l}
 \newcommand{\bm}{\boldsymbol m}
 \newcommand{\bn}{\boldsymbol n}
 \newcommand{\bgamma}{\boldsymbol \gamma}
 \newcommand{\blambda}{\boldsymbol \lambda}
 \newcommand{\bq}{\boldsymbol q}
 \newcommand{\bt}{\boldsymbol t}
 \newcommand{\bu}{\boldsymbol u}
 \newcommand{\bv}{\boldsymbol v}
 \newcommand{\bw}{\boldsymbol w}
 \newcommand{\bx}{\boldsymbol x}
 \newcommand{\bwy}{\boldsymbol y}
 \newcommand{\bnu}{\boldsymbol \nu}
 \newcommand{\bxi}{\boldsymbol \xi}
 \newcommand{\bz}{\boldsymbol z}
 \newcommand{\whG}{\widehat{G}}
 \newcommand{\oK}{\overline{K}}
 \newcommand{\oKt}{\overline{K}^{\times}}
 \newcommand{\oQ}{\overline{\Q}}
 \newcommand{\oq}{\oQ^{\times}}
 \newcommand{\oQt}{\oQ^{\times}/\Tor\bigl(\oQ^{\times}\bigr)}
 \newcommand{\ot}{\Tor\bigl(\oQ^{\times}\bigr)}
 \newcommand{\h}{\frac12}
 \newcommand{\hh}{\tfrac12}
 \newcommand{\dx}{\text{\rm d}x}
 \newcommand{\dbx}{\text{\rm d}\bx}
 \newcommand{\dy}{\text{\rm d}y}
 \newcommand{\dmu}{\text{\rm d}\mu}
 \newcommand{\dnu}{\text{\rm d}\nu}
 \newcommand{\dla}{\text{\rm d}\lambda}
 \newcommand{\dlav}{\text{\rm d}\lambda_v}
 \newcommand{\trho}{\widetilde{\rho}}
 \newcommand{\dtrho}{\text{\rm d}\widetilde{\rho}}
 \newcommand{\drho}{\text{\rm d}\rho}

\title[Minkowski]{Minkowski's theorem on independent\\conjugate units}
\author{Shabnam Akhtari and Jeffrey~D.~Vaaler}
\thanks{}
\medskip

\address{Department of Mathematics, University of Oregon, Eugene, Oregon 97402 USA}
\email{akhtari@uoregon.edu}
\medskip

\address{Department of Mathematics, University of Texas, Austin, Texas 78712 USA}
\email{vaaler@math.utexas.edu}

\begin{abstract}
We call a unit $\beta$ in a finite, Galois extension $l/\Q$ a Minkowski unit if the subgroup generated by 
$\beta$ and its conjugates over $\Q$ has maximum rank in the unit group of $l$. Minkowski showed the existence of such 
units in every Galois extension.  We give a new proof of Minkowski's theorem and show that there exists a
Minkowski unit $\beta \in l$ such that the Weil height of $\beta$ is comparable with the sum of the heights of a fundamental 
system of units for $l$.  Our proof implies a bound on the index of the subgroup generated by the algebraic conjugates 
of $\beta$ in the unit group of $l$. 

If $k$ is an intermediate field such that
\begin{equation*}\label{first0}
\Q \subseteq k \subseteq l,
\end{equation*}
and $l/\mathbb{Q}$ and  $k/\Q$ are Galois extensions, we prove an analogous bound for the subgroup of relative units. In 
order to establish our results for relative units, a number of new ideas are combined with techniques from the geometry 
of numbers and the Galois action on places.
\end{abstract}
\maketitle
\numberwithin{equation}{section}

 \section{Introduction}

Let $l$ be an algebraic number field, $O_l$ the ring of algebraic integers in $l$, and $O_l^{\times}$ the multiplicative
group of units in $O_l$.  We write $\Tor\bigl(O_l^{\times}\bigr)$ for the torsion subgroup of $O_l^{\times}$, which is
the finite group of roots of unity in $O_l^{\times}$.  Dirichlet's unit theorem asserts that there exists a nonnegative integer $r = r(l)$,
called the {\it rank} of $O_l^{\times}$, such that
\begin{equation*}\label{first1}
O_l^{\times} \cong \Tor\bigl(O_l^{\times}\bigr) \oplus \Z^r.
\end{equation*}
Alternatively, there exists a finite collection $\eta_1, \eta_2, \dots , \eta_r$ of units such that each unit $\alpha$ in $O_l^{\times}$
has a unique representation as
\begin{equation*}\label{first4}
\alpha = \zeta \eta_1^{m_1} \eta_2^{m_2} \cdots \eta_r^{m_r},
\end{equation*}
where $\zeta$ is an element of the torsion subgroup and $m_1, m_2, \dots, m_r$ are integers.  We say that
$\eta_1, \eta_2, \dots , \eta_r$ form a {\it fundamental system} of units for $O_l^{\times}$.  In this
paper we assume that $l$ is not $\Q$, and that $l$ is not an imaginary, quadratic extension of $\Q$.  This hypothesis
insures that the rank $r = r(l)$ is positive.

If $l/\Q$ is a Galois extension, we write $\Aut(l/\Q)$ for the Galois group of all automorphisms of $l$ that fix $\Q$.
And if $\beta$ is a unit in $O_l^{\times}$, we write
\begin{equation}\label{first7}
\langle \sigma(\beta) : \sigma \in \Aut(l/\Q)\rangle
\end{equation}
for the subgroup of $O_l^{\times}$ generated by the collection of all Galois conjugates of $\beta$.  Minkowski \cite{minkowski00}
(see also \cite[Theorem 3.26]{narkiewicz2010})
showed that it is always possible to select $\beta$ so that the subgroup (\ref{first7}) has rank $r$, which is obviously the maximum 
possible rank of a subgroup.  We call a unit $\beta$ in a Galois extension a {\it Minkowski unit} if the subgroup (\ref{first7})
generated by its conjugates over $\Q$ has maximum rank.  Thus a unit $\beta$ in a Galois extension $l/\Q$ is a Minkowski
unit if and only if there exists a subset of its conjugates
\begin{equation*}\label{first8}
\big\{\sigma(\beta) : \sigma \in \Aut(l/\Q)\big\}
\end{equation*}
that has cardinality $r$ and is multiplicatively independent.  One of our objectives is to show that there exists a
Minkowski unit $\beta$ such that the Weil height of $\beta$ is comparable with the sum of the heights of a fundamental 
system of units.  If $k$ is an intermediate field such that
\begin{equation*}\label{first9}
\Q \subseteq k \subseteq l,
\end{equation*}
and $k/\Q$ is a Galois extension, we prove an analogous bound for the subgroup of relative units.

Because roots of unity play no significant role in the results we establish, it will be convenient to work with an equivalent
concept of Minkowski unit in the torsion free abelian group
\begin{equation}\label{first10}
F_l = O_l^{\times}/\Tor\bigl(O_l^{\times}\bigr).
\end{equation}
It follows that $O_l^{\times}$ and $F_l$ are both finitely generated abelian groups of the same rank $r = r(l)$, but $F_l$ is 
free abelian.   If
\begin{equation*}\label{first11}
\alpha \mapsto \alpha \Tor\bigl(O_l^{\times}\bigr)
\end{equation*}
is the canonical homomorphism, then the results we prove about subsets of $F_l$ immediately imply corresponding 
results for subsets of $O_l^{\times}$ by considering inverse images.  If $l/\Q$ is a Galois 
extension, then $\Aut(l/\Q)$ acts on $O_l^{\times}$, and also acts on the subgroup $\Tor\bigl(O_l^{\times}\bigr)$.
It follows that each automorphism $\sigma$ in $\Aut(l/\Q)$ induces a well defined automorphism on the quotient group $F_l$.
That is, if $\sigma$ is an automorphism in $\Aut(l/\Q)$ and $\alpha$ is an element of $O_l^{\times}$, then
\begin{equation*}\label{first12}
(\sigma, \alpha) \mapsto \sigma(\alpha)
\end{equation*}  
defines an action of $\Aut(l/\Q)$ on $O_l^{\times}$, and
\begin{equation*}\label{first13}
\bigl(\sigma, \alpha \Tor\bigl(O_l^{\times}\bigr)\bigr) \mapsto \sigma(\alpha) \Tor\bigl(O_l^{\times}\bigr)
\end{equation*}
defines a corresponding action of $\Aut(l/\Q)$ on cosets in $F_l$.  Then it is obvious that the image of the subgroup
(\ref{first7}) in $F_l$, is the subgroup
\begin{equation}\label{first14}
\langle \sigma(\beta) \Tor\bigl(O_l^{\times}\bigr) : \sigma \in \Aut(l/\Q) \rangle.
\end{equation}
Moreover, if $\sigma_1, \sigma_2, \dots , \sigma_r$ are distinct automorphisms in $\Aut(l/\Q)$, then the subset
\begin{equation*}\label{second1}
\big\{\sigma_1(\beta), \sigma_2(\beta), \dots , \sigma_r(\beta)\big\}
\end{equation*}
is multiplicatively independent in $O_l^{\times}$ if and only if the subset
\begin{equation*}\label{second2}
\big\{\sigma_1(\beta) \Tor\bigl(O_l^{\times}\bigr), \sigma_2(\beta) \Tor\bigl(O_l^{\times}\bigr), 
			\dots , \sigma_r(\beta) \Tor\bigl(O_l^{\times}\bigr)\big\}
\end{equation*}
is multiplicatively independent in $F_l$.  Thus we may speak of a {\it Minkowski unit} in $F_l$,
by which we understand a coset $\beta \Tor\bigl(O_l^{\times}\bigr)$ in $F_l$ such that the subgroup (\ref{first14})
has rank $r = r(l)$.  However to simplify notation, in the remainder of the paper we usually write elements of $F_l$ as coset 
representatives rather than as cosets.

We write $l^{\times}$ for the multiplicative group of nonzero elements in $l$, and 
\begin{equation*}\label{sum1}
h : l^{\times} \rightarrow [0, \infty)
\end{equation*}
for the absolute, logarithmic Weil height.  This height is defined later in (\ref{int3}).  As is well known, if $\alpha$ and $\zeta$
belong to $l^{\times}$ and $\zeta$ is a root of unity, then we have $h(\alpha \zeta) = h(\alpha)$.  It follows that $h$ is
constant on cosets of the quotient group 
\begin{equation*}\label{sum3}
\G_l = l^{\times}/\Tor\bigl(l^{\times}\bigr) = l^{\times}/\Tor\bigl(O_l^{\times}\bigr),
\end{equation*}
and therefore $h$ is well defined as a map
\begin{equation}\label{sum5}
h : \G_l \rightarrow [0, \infty).
\end{equation}
Elementary properties of the Weil height imply that the map
\begin{equation*}\label{sum7}
(\alpha, \beta) \mapsto h\bigl(\alpha \beta^{-1}\bigr)
\end{equation*}
defines a metric on $\G_l$.  In this work we will only have occasion to use the height $h$ on the subgroup $F_l \subseteq \G_l$. 
Further properties of the Weil height on groups are discussed in \cite{akhtari2015}, \cite{bombieri2006}, and 
\cite{vaaler2014}.

It follows from Minkowski's work in \cite{minkowski00} that if $F_l$ has positive 
rank $r  = r(l)$, then there exists a coset representative $\beta$ in $F_l$ such that the multiplicative subgroup
\begin{equation}\label{first18}
\bB = \langle \sigma(\beta) : \sigma \in \Aut(l/\Q)\rangle \subseteq F_l
\end{equation}
generated by the orbit of $\beta$ also has rank $r$.  That is, $\beta$ is a Minkowski unit in $F_l$, and therefore the index 
$[F_l : \bB]$ is finite.  Here we give a proof of Minkowski's theorem which includes a bound on the index $[F_l : \bB]$, 
and a bound on the absolute logarithmic Weil height $h(\beta)$ of the Minkowski unit $\beta$.  

\begin{theorem}\label{maintheorem1} Let $l/\mathbb{Q}$ be a Galois extension, and assume that $l$ has positive unit rank $r = r(l)$.
Let $\eta_1, \eta_2, \dots , \eta_r$ be multiplicatively independent units in $F_l$, and write
\begin{equation*}\label{sum9}
\aA = \langle \eta_1, \eta_2, \dots , \eta_r\rangle \subseteq F_l
\end{equation*}
for the subgroup of rank $r$ that they generate.   Then there exists a  Minkowski unit $\beta$ contained in $\aA$ such that
\begin{equation}\label{first16}
h(\beta) \le 2 \sum_{j=1}^r h(\eta_j).
\end{equation}
Moreover, the subgroup {\rm (\ref{first18})} has index bounded by
\begin{equation}\label{first19}
\Reg(l) [F_l : \bB] \le \bigl([l : \Q] h(\beta)\bigr)^r,
\end{equation}
where $\Reg(l)$ is the regulator of $l$.
\end{theorem}

Theorem \ref{maintheorem1} is a simplified version of our more elaborate Theorem \ref{thmspecial1}.  In section \ref{SMU} we 
define {\it special Minkowski units} in $F_l$ with respect to a given archimedean place $\widehat{w}$ of the field $l$.  Lemma \ref{lemspecial2} shows that a special Minkowski unit is a Minkowski unit.  Then Theorem \ref{thmspecial1} establishes 
the existence of a special Minkowski unit $\beta$ with respect to a given archimedean place $\widehat{w}$,
that satisfies the conclusion of Theorem \ref{maintheorem1}, and has additional properties.  For example,
Theorem \ref{thmspecial1} identifies explicit subsets of the orbit of $\beta$ that are multiplicatively independent.  This is
straightforward if $l/\Q$ is a totally real Galois extension, but more complicated if $l/\Q$ is a totally complex 
Galois extension.

If $k$ is an intermediate field, that is, if
\begin{equation*}\label{first23}
\Q \subseteq k \subseteq l,
\end{equation*}
then the norm from $l^{\times}$ into $k^{\times}$ induces a homomorphism
\begin{equation}\label{first27}
\norm_{l/k} : F_l \rightarrow F_k.
\end{equation}
The kernel of the homomorphism (\ref{first27}) is the subgroup
\begin{equation*}\label{first31}
E_{l/k} = \big\{\alpha \in F_l : \norm_{l/k}(\alpha) = 1\big\}
\end{equation*}
of relative units in $F_l$.  (See \cite{akhtari2015}, \cite{costa1991}, and \cite{costa1993} for further properties of this subgroup.)
Because $F_k$ is a free abelian group, the kernel $E_{l/k}$  of the homomorphism (\ref{first27}) is a direct sum in $F_l$.  
And it follows from the discussion in section \ref{RUI} that
\begin{equation}\label{first36}
r(l/k) = \rank E_{l/k} = r(l) - r(k).
\end{equation}
Therefore $E_{l/k}$ is a proper subgroup of $F_l$ if and only if $1 \le r(k) < r(l)$. 

In general the Galois group $\Aut(l/\Q)$ does not act on the subgroup $E_{l/k}$, and therefore the simplest analogue of Minkowski's 
theorem cannot hold in $E_{l/k}$.  If we assume that both $l/\Q$ and $k/\Q$ are Galois extensions, or equivalently, if we assume that 
$\Aut(l/k)$ is a normal subgroup of $\Aut(l/\Q)$, then we show in Lemma \ref{lemfour2} that the group $\Aut(l/\Q)$ does act on the subgroup
$E_{l/k}$ of relative units.  If both $l/\Q$ and $k/\Q$ are Galois extensions, 
we say that an element $\gamma$ in $E_{l/k}$ is a {\it relative Minkowski unit} for the subgroup $E_{l/k}$, if the subgroup
\begin{equation*}\label{first43}  
\cC = \langle \sigma(\gamma) : \sigma \in \Aut(l/\Q)\rangle \subseteq E_{l/k}
\end{equation*}
generated by the orbit of $\gamma$ has rank equal to $r(l/k)$.  Obviously (\ref{first36}) implies that $r(l/k)$ is the maximum
possible rank of a subgroup in $E_{l/k}$.  Our second main result establishes the existence of a relative Minkowski 
unit in $E_{l/k}$, and includes a bound on the height that is analogous to (\ref{first16}).

\begin{theorem}\label{maintheorem2}  Let $l/\Q$ and $k/\Q$ be finite, Galois extensions such that
\begin{equation*}\label{first47}
\Q \subseteq k \subseteq l,\quad\text{and}\quad 1 \le r(k) < r(l),
\end{equation*}
where $r(k)$ is the rank of $F_k$, and $r(l)$ is the rank of $F_l$.
Let $\eta_1, \eta_2, \dots , \eta_{r(l)}$ be a basis for the group $F_l$.
\begin{itemize}
\item[(i)]  If $l/\Q$ is a totally real Galois extension, then there exists a relative Minkowski unit $\gamma$ in $E_{l/k}$
such that
\begin{equation}\label{first52}
h(\gamma) \le 4\bigl([l : k] - 1\bigr) \sum_{j = 1}^{r(l)} h(\eta_j).
\end{equation}
\item[(ii)]  If $l/\Q$ is a totally complex Galois extension, then there exists a relative Minkowski unit $\gamma$ in $E_{l/k}$
such that
\begin{equation}\label{first55}
h(\gamma) \le 8\bigl([l : k] - 1\bigr) \sum_{j = 1}^{r(l)} h(\eta_j).
\end{equation}
\end{itemize}
\end{theorem}

We recall that a Galois extension $l/\Q$ is either totally real or totally complex, and therefore (i) and (ii) in the statement of
Theorem \ref{maintheorem2} cover all cases.

Theorem \ref{maintheorem2} is a simplified version of Theorem \ref{thmmink1}, which provides more precise information
about the construction of the relative unit $\gamma$.  A brief outline of the proof of Theorem \ref{thmmink1}, in case $l/\Q$ is a totally 
real Galois extension, is as follows.  We begin with a special Minkowski unit $\beta$ contained in the group $F_l$, which exists by 
Theorem \ref{thmspecial1}, and satisfies the inequality (\ref{first16}).  In section \ref{RUIII} we define a bi-homomorphsim
\begin{equation*}\label{first60}
\Delta : F_l \times \Z^N \rightarrow F_l,
\end{equation*}
where $N$ is the the number of archimedean places of $l$.  Corollary \ref{cordelta1} asserts that if $\beta$ is a special
Minkowski unit, and 
\begin{equation*}\label{first64}
\{f_1, f_2, \dots f_R\},\quad\text{where $R = r(l/k)$,} 
\end{equation*}
is a collection of linearly independent elements in a certain subgroup $\eL_{l/k} \subseteq \Z^N$, then the units in the set
\begin{equation}\label{first68}
\big\{\Delta(\beta, f_1), \Delta(\beta, f_2), \dots , \Delta(\beta, f_R)\big\}
\end{equation}
are multiplicatively independent relative units in $E_{l/k}$.  Then we appeal to results in section \ref{FFG} that establish
the existence of an element $\lambda$ in $\eL_{l/k}$, and a subset 
\begin{equation*}\label{first70}
\{\psi_1, \psi_2, \dots , \psi_R\} \subseteq \Aut(l/\Q),
\end{equation*}
such that the collection of conjugates
\begin{equation*}\label{first72}
\big\{\psi_1\bigl(\Delta(\beta, \lambda)\bigr), \psi_2\bigl(\Delta(\beta, \lambda)\bigr), 
	\dots , \psi_R\bigl(\Delta(\beta, \lambda)\bigr)\big\}
\end{equation*}
has exactly the shape of the subset (\ref{first68}).  It follows that
\begin{equation*}\label{first76}
\gamma = \Delta(\beta, \lambda)
\end{equation*}
is a relative Minkowski unit for the subgroup $E_{l/k}$.  The bound (\ref{first52}) follows from (\ref{first16}), the definition of 
$\lambda$ in $\eL_{l/k}$, and the definition of the bi-homomorphism $\Delta$.

If $l/\Q$ is a totally complex Galois extension the argument is similar.  In this case we use a special Minkowski 
unit of the form $\beta \rho(\beta)$, where $\rho$ is a certain element of $\Aut(l/\Q)$ of order $2$.  Then we show that
\begin{equation*}\label{first79}
\gamma = \Delta(\beta \rho(\beta), \lambda).
\end{equation*}
is a relative Minkowski unit for the subgroup $E_{l/k}$.  The bound (\ref{first55}) on the height of $\gamma$ follows as in
the previous case.  

\section{Preliminary lemmas}

In this section we prove four elementary lemmas about real matrices.

\begin{lemma}\label{lemgen1} Let $A= \bigl(a_{mn}\bigr)$ be a real, nonsingular, $N \times N$ matrix.  Then there exists a point
\begin{equation*}
\bxi = \begin{pmatrix}   \xi_1 \\
                                    \xi_2 \\
                                    \vdots \\
                                    \xi_N \end{pmatrix}
\end{equation*}
in $\Z^N$ such that
\begin{equation}\label{galois7}
0 < \sum_{n=1}^N a_{mn}\xi_n \le \sum_{n=1}^N |a_{mn}|\quad\text{for each  $m=1, 2, \dots , N$.}
\end{equation}
\end{lemma}

\begin{proof} For each positive integer $l$, define the (column) vector $\bu^{(l)}$ in $\R^N$ by
\begin{equation}\label{galois9}
u_m^{(l)} = (\tfrac12 + \tfrac1l)\sum_{n=1}^N |a_{mn}|\quad\text{for\ each $m=1, 2, \dots , N$.}
\end{equation}
Then set $\bv^{(l)} = A^{-1}\bu^{(l)}$, and select $\bxi^{(l)}$ in $\Z^N$ so that
\begin{equation*}\label{galois11}
-\hh \le v_n^{(l)} - \xi_n^{(l)} \le \hh\quad\text{for\ each $n=1, 2, \dots , N$.}
\end{equation*}
It follows that 
\begin{equation}\label{galois13}
\biggl|\sum_{n=1}^N a_{mn}\xi_n^{(l)} - u_m^{(l)}\biggr| = \biggl|\sum_{n=1}^N a_{mn}\bigl(\xi_n^{(l)} - v_n^{(l)}\bigr)\biggr| 
	\le \tfrac12 \sum_{n=1}^N |a_{mn}|
\end{equation}
for each $m = 1, 2, \dots , N$.  Combining (\ref{galois9}) and (\ref{galois13}) we get
\begin{equation*}\label{galois14}
 -\tfrac12 \sum_{n=1}^N |a_{mn}| \le  \sum_{n=1}^N a_{mn}\xi_n^{(l)} - (\tfrac12 + \tfrac1l)\sum_{n=1}^N |a_{mn}|        
 		\le \tfrac12 \sum_{n=1}^N |a_{mn}|,
\end{equation*}
and therefore
\begin{equation}\label{galois15}
\tfrac1l \sum_{n=1}^N |a_{mn}| \le \sum_{n=1}^N a_{mn}\xi_n^{(l)} \le (1 + \tfrac1l)\sum_{n=1}^N |a_{mn}|
\end{equation}
for each $m=1, 2, \dots, N.$

Now observe that the set
\begin{equation*}\label{galois17}
\Big\{\bx\in\R^N: 0 < \sum_{n=1}^N a_{mn}x_n \le 2\sum_{n=1}^N |a_{mn}|\ \ \text{for $m=1, 2, \dots , N$}\Big\}
\end{equation*}
is bounded, and therefore it intersects $\Z^N$ in only finitely many points.  It follows from (\ref{galois15}) that the 
map $l\mapsto \bxi^{(l)}$ is constant for $l$ in an infinite subset of $\{1, 2, 3, \dots \}$.  Letting $\bxi$ in $\Z^N$ denote 
such a constant, we conclude that
\begin{equation*}\label{galois19}
0 < \sum_{n=1}^N a_{mn}\xi_n \le \sum_{n=1}^N |a_{mn}|\quad\text{for\ each $m=1, 2,  \dots , N$,}
\end{equation*}
and the lemma plainly follows.
\end{proof}

The following is a variant of a lemma due to Minkowski.

\begin{lemma}\label{lemgen2}  Let $A = \bigl(a_{mn}\bigr)$ be a real, $N \times N$ matrix such that
\begin{equation}\label{galois39}
0 < \sum_{m = 1}^N a_{mn},\quad\text{for each $n = 1, 2, \dots , N$},
\end{equation}
and
\begin{equation*}\label{galois41}
a_{mn} < 0,\quad\text{for $m \not= n$}.
\end{equation*}
Then $A$ is nonsingular.
\end{lemma} 

\begin{proof}  Assume that $\det A = 0$.  Then there exists a point $\bx \not= \bo$ in $\R^N$ such that
\begin{equation}\label{galois45}
0 = \sum_{m = 1}^N a_{mn} x_m,\quad\text{for each $n = 1, 2, \dots , N$}.
\end{equation}
By replacing $\bx$ with $-\bx$ if necessary, we can select an integer $r$ such that
\begin{equation}\label{galois49}
0 < x_r = \max\{|x_n| : n = 1, 2, \dots , N\}.
\end{equation}
Then applying (\ref{galois45}) with $n = r$, and using (\ref{galois49}), we find that
\begin{align*}\label{galois53}
\begin{split}
0 &= \sum_{m = 1}^N a_{m r} x_m = a_{r r} x_r + \sum_{\substack{m = 1\\m \not= r}}^N a_{m r} x_m\\
   &\ge a_{r r} x_r + \sum_{\substack{m = 1\\m \not= r}}^N a_{m r} x_r = x_r \sum_{m = 1}^N a_{m r} > 0.
\end{split}
\end{align*}
We conclude that $\bx \not= \bo$ does not exist, and therefore $A$ is nonsingular.
\end{proof}

If $A$ is a real, $N \times N$ matrix then the $\Q$-rank of $A$ is the number of $\Q$-linearly independent rows (or columns) 
of $A$.  Similarly, the $\R$-rank of $A$ is the number of $\R$-linearly independent rows (or columns) of $A$.
In general the $\Q$-rank of $A$ is greater than or equal to the $\R$-rank of $A$.  Of course $\det A \not= 0$ if and only if the 
$\R$-rank of $A$ is $N$, and this implies that the $\Q$-rank is also $N$.  But the matrix
\begin{equation*}\label{gen1}
A = \begin{pmatrix} \sqrt{2}& 1\\
                                        2 &  \sqrt{2} \end{pmatrix}
\end{equation*}  
has $\Q$-rank equal to $2$, and $\R$-rank equal to $1$.

\begin{lemma}\label{lemgen3}  Let $A = \bigl(a_{mn}\bigr)$ be a real, $N \times N$ matrix such that
\begin{equation}\label{gen3}
\sum_{m = 1}^N a_{mn} = 0,\quad\text{for each $n = 1, 2, \dots , N$},
\end{equation}
and
\begin{equation}\label{gen6}
\sum_{n = 1}^N a_{mn} = 0,\quad\text{for each $m = 1, 2, \dots , N$}.
\end{equation}
Let $A_{(m,n)}$ denote the $(N-1) \times (N-1)$ submatrix of $A$ obtained by removing the row indexed by $m$, and
removing the column indexed by $n$.  Then there exists a real constant $c$ such that
\begin{equation}\label{gen8}
(-1)^{m+n} \det A_{(m,n)} = c
\end{equation}
for each pair $(m,n)$.  Moreover, we have $c \not= 0$ if and only if the $\R$-rank of $A$ and the $\Q$-rank of
$A$ are both equal to $N-1$.
\end{lemma}

\begin{proof}  It is clear from (\ref{gen3}) that the $\Q$-rank of $A$ is less than or equal to $N-1$, and $\det A = 0$.  
If the $\R$-rank is less than or equal to $N-2$ then (\ref{gen8}) holds with $c = 0$.  Thus we assume for the 
remainder of the proof that the $\R$-rank of $A$, and the $\Q$-rank of $A$, are both equal to $N-1$.  Then it follows from 
(\ref{gen3}) that 
\begin{equation}\label{galois1}
\sum_{m = 1}^N a_{mn} x_m = 0,\quad\text{for each $n = 1, 2, \dots , N$},
\end{equation}
if and only if $m \mapsto x_m$ is constant, and it follows from (\ref{gen6}) that
\begin{equation}\label{galois3}
\sum_{n = 1}^N a_{mn} y_n = 0,\quad\text{for each $m = 1, 2, \dots , N$},
\end{equation}
if and only if $n \mapsto y_n$ is constant.  We also have 
\begin{equation*}\label{galois20}
(-1)^{i + j}\det A_{(i,j)} \not= 0
\end{equation*}
for some pair of indices $(i,j)$. 

Recall that the Laplace expansion of the determinant along columns is
\begin{equation}\label{galois21}
\sum_{m=1}^N (-1)^{m+t} a_{ms}\det A_{(m, t)} = \begin{cases} \det A& \text{if $s = t$,}\\
                                                                                                             0& \text{if $s \not= t$.}\end{cases}
\end{equation}
As $\det A = 0$, it follows from (\ref{galois1}) and (\ref{galois21}) that for each index $t$ the function
\begin{equation*}\label{galois23}
m \mapsto (-1)^{m + t} \det A_{(m, t)} = c(t)
\end{equation*}
is a real constant that depends only on $t$.  Similarly, the expansion along rows is
\begin{equation}\label{galois26}
\sum_{n=1}^N (-1)^{v + n} a_{un}\det A_{(v, n)} = \begin{cases} \det A& \text{if $u = v$,}\\
                                                                                                         0& \text{if $u \not= v$.}\end{cases}
\end{equation}
Again we have $\det A = 0$, and therefore (\ref{galois3}) and (\ref{galois26}) imply that for each index $v$ the function
\begin{equation*}\label{galois29}
n \mapsto (-1)^{v + n} \det A_{(v, n)} = d(v)
\end{equation*}
is a real constant that depends only on $v$.  We have shown that
\begin{equation}\label{galois31}
c(n) = (-1)^{m + n} \det A_{(m, n)} = d(m)
\end{equation}
for each pair of indices $(m, n)$.  It follows from (\ref{galois31}) that both
\begin{equation*}\label{galois33}
m \mapsto d(m),\quad\text{and}\quad n \mapsto c(n),
\end{equation*}
are constant.  Taking $i = m$ and $j = n$ shows that this constant is 
\begin{equation*}\label{galois35}
(-1)^{i + j} \det A_{(i, j)} \not= 0.
\end{equation*}
This proves the lemma.
\end{proof}

Lemma \ref{lemgen2} and Lemma \ref{lemgen3} can be combined to establish the following general result.

\begin{lemma}\label{lemgen4}  Let $A = \bigl(a_{mn}\bigr)$ be a real, $N \times N$ matrix such that
\begin{equation}\label{gen201}
0 = \sum_{m = 1}^N a_{mn},\quad\text{for each $n = 1, 2, \dots , N$},
\end{equation}
and
\begin{equation}\label{gen203}
a_{mn} < 0,\quad\text{for $m \not= n$}.
\end{equation}
Then $A$ satisfies the following conditions.
\begin{itemize}
\item[(i)]  The $\Q$-$\rank$ of $A$, and the $\R$-$\rank$ of $A$, are both equal to $N-1$.
\item[(ii)]  If $\bwy \not= \bo$ is a point in $\R^N$ such that
\begin{equation}\label{gen204}
0 = \sum_{n = 1}^N a_{mn} y_n,\quad\text{for each $m = 1, 2, \dots , N$},
\end{equation}
then the co-ordinates $y_1, y_2, \dots , y_N$ are all positive, or all negative.
\item[(iii)]  For each pair $(m, n)$ the submatrix $A_{(m, n)}$ is nonsingular.
\end{itemize}
\end{lemma} 

\begin{proof}  It is clear from (\ref{gen201}) that the rows of $A$ are $\Q$-linearly dependent, and therefore
the $\Q$-rank of $A$ is at most $N-1$.  Thus it suffices to show that the $\R$-rank of $A$ is 
at least $N-1$.  Let $\bu$ in $\Z^N$ be such that $u_n = 1$ for each $n = 1, 2, \dots , N$, and let $\nN$ be the null space
\begin{equation*}\label{gen207}
\nN = \big\{\bx \in \R^N : \bx^T A = \bo^T\big\}.
\end{equation*}
Then $\nN$ is an $\R$-linear subspace of $\R^N$, and by (\ref{gen201}) the subspace $\nN$ contains $\bu$.  Assume that 
$\bw \not= \bo$ is also in $\nN$.  Let
\begin{equation*}\label{gen209}
w_i = \min\{w_m : m = 1, 2, \dots, N\},\quad\text{and}\quad w_j = \max\{w_m : m = 1, 2, \dots , N\},
\end{equation*}
and assume that $w_i < w_j$.  Then $\bw - w_i \bu$ belongs to $\nN$, and therefore
\begin{equation}\label{gen211}
0 = \sum_{\substack{m = 1\\m \not= i}}^N (w_m - w_i) a_{mn},\quad\text{for each $n = 1, 2, \dots , N$}.
\end{equation}
Taking $n = i$ in (\ref{gen211}), we find that
\begin{equation*}\label{gen212}
0 = \sum_{\substack{m = 1\\m \not= i}}^N (w_m - w_i) a_{mi} \le (w_j - w_i) a_{ji} < 0,
\end{equation*} 
which is impossible.  Therefore $w_i = w_j$, and $\bw$ is a real multiple of $\bu$.  Hence $\nN$ is
an $\R$-linear subspace of dimension $1$.  It follows that the $\R$-rank of $A$ is $N-1$.

Let $1 \le r \le N$ and let $A_{(r, r)}$ be the $(N-1) \times (N-1)$ submatrix of $A$ obtained by removing the row indexed 
by $r$, and the column indexed by $r$.  It follows from (\ref{gen201}) and (\ref{gen203}) that $A_{(r, r)}$ satisfies the
hypotheses of Lemma \ref{lemgen2}, but with $N$ replaced by $N-1$.  We conclude from Lemma \ref{lemgen2} 
that each submatrix $A_{(r, r)}$ is nonsingular.  Now suppose that $\bwy \not= \bo$ satisfies (\ref{gen204}), and $y_r = 0$.
Let $\bz$ be the (column) vector in $\R^{N-1}$ obtained from $\bwy$ by removing $y_r = 0$.  Then we have
\begin{equation*}\label{gen213}
\bz \not= \bo,\quad\text{and}\quad A_{(r, r)} \bz = \bo,
\end{equation*}
which contradicts the fact that the submatrix $A_{(r, r)}$ is nonsingular.  We have shown that if $\bwy \not= \bo$ satisfies
(\ref{gen204}), then $y_r \not= 0$ for each $r = 1, 2, \dots , N$.  Next we define
\begin{equation*}\label{gen214}
I = \{n : 0 < y_n\},\quad\text{and}\quad J = \{n : y_n < 0\},
\end{equation*}
and we assume that both $I$ and $J$ are not empty.  It follows from (\ref{gen201}) and (\ref{gen203}) that
\begin{equation}\label{gen215}
0 < \sum_{m \in I} a_{mn},\quad\text{if $n$ belongs to $I$,}
\end{equation}
and
\begin{equation}\label{gen216}
\sum_{m \in I} a_{mn} < 0,\quad\text{if $n$ belongs to $J$.}
\end{equation}
Applying (\ref{gen204}), (\ref{gen215}), and (\ref{gen216}), we find that
\begin{align*}\label{gen217}
\begin{split}
0 &= \sum_{m \in I} \sum_{n = 1}^N a_{mn} y_n\\
   &= \sum_{n \in I} \biggl(\sum_{m \in I} a_{mn} \biggr) y_n + \sum_{n \in J}\biggl(\sum_{m \in I} a_{mn} \biggr) y_n\\
   &> 0.
\end{split}
\end{align*}
The contradiction implies that either $I$ or $J$ is empty.  That is, the co-ordinates $y_1, y_2, \dots , y_N$ are all
positive, or all negative.

We continue to suppose that $\bwy \not= \bo$ is a point in $\R^N$ that satisfies (\ref{gen204}).  Then we let
$[y_n]$ denote the $N \times N$ diagonal matrix with $y_1, y_2, \dots , y_N$ as consecutive diagonal entries.  We have 
already verified (ii), and therefore
\begin{equation}\label{gen223}
Y = \det [y_n] = y_1 y_2 \cdots y_N \not= 0.
\end{equation}
Let $B$ denote the $N \times N$ real matrix
\begin{equation}\label{gen227}
B = A[y_n] = \bigl(a_{mn} y_n\bigr).
\end{equation}
From (\ref{gen201}) we get
\begin{equation*}\label{gen231}
0 = \sum_{m = 1}^N a_{mn} y_n,\quad\text{for each $n = 1, 2, \dots , N$},
\end{equation*}
and (\ref{gen204}) asserts that
\begin{equation*}\label{gen235}
0 = \sum_{n = 1}^N a_{mn} y_n,\quad\text{for each $m = 1, 2, \dots , N$}.
\end{equation*}
It follows that $B$ satisfies the hypotheses of Lemma \ref{lemgen3}.  As $[y_n]$ is nonsingular, $B$ has rank $(N-1)$. 
Hence there exists a constant $b \not= 0$ such that
\begin{equation*}\label{gen239}
(-1)^{m+n} \det B_{(m, n)} = b
\end{equation*}
for each $(N-1) \times (N-1)$ submatrix $B_{(m, n)}$.  Using (\ref{gen223}) and (\ref{gen227}), we find that
\begin{equation*}\label{gen243}
y_n \det B_{(m, n)} = Y \det A_{(m, n)}
\end{equation*}
for each integer pair $(m, n)$.  This shows that each submatrix $A_{(m, n)}$ is nonsingular, and also establishes the identity
\begin{equation*}\label{gen247}
(-1)^{m+n} Y \det A_{(m, n)} = b y_n
\end{equation*}
for each integer pair $(m, n)$ and a real constant $b \not= 0$.
\end{proof}

\section{Functions on finite groups}\label{FFG}

Throughout this section we assume that $G$ is a finite group with subgroups $H \subseteq G$ and $K \subseteq G$.  
We assume that $H$ is a normal subgroup of $G$, and we write $N = [G : K]$ for the index of $K$ in 
$G$.  We consider two cases: either $H \cap K = \{1\}$, or $K \subseteq H$.  If $H \cap K = \{1\}$, then 
\begin{equation*}\label{map0}
H K = K H = \{h k : \text{$h \in H$ and $k \in K$}\},
\end{equation*}
and the map
\begin{equation*}\label{map1}
(h, k) \mapsto h k
\end{equation*}
from $H \times K$ into $H K$ is bijective.  Because $H$ is normal in $G$, the subset $H K$ is a subgroup of $G$, (see 
\cite[Chapter 2, Proposition (8.6)]{artin1991}).  If $K \subseteq H$ the situation is simpler because $H K = H$.  In both cases we define
\begin{equation*}\label{map2}
I = [G : HK],\quad\text{and}\quad J = [HK : K],
\end{equation*}
so that $I J = N$. 

We write $\Z^N$ for the free abelian group of rank $N$, and we identify elements of this group with functions
\begin{equation*}\label{map3}
f : G \rightarrow \Z
\end{equation*}
that are constant on each left coset of $K$.  As $K$ has $N$ left cosets in $G$, it is clear that this group of functions {\it is}
free abelian of rank $N$.  Alternatively, we write $G/K$ for the collection of all left cosets of $K$ in $G$, and we identify
elements of $\Z^N$ with functions
\begin{equation*}\label{map4} 
f : G/K \rightarrow \Z.
\end{equation*}

As is well known, the group $G$ acts on the set $G/K$ of all left cosets of $K$ in $G$ by multiplication on the left,
(see \cite[Chapter 5, section 6 and section 7]{artin1991}).  This action induces an action
of $G$ on $\Z^N$ as follows: if $g$ belongs to $G$ and $x \mapsto f(x)$ is a function in $\Z^N$, we write 
$[g, f]$ for the action of $g$ on $f$, and we define this element of $\Z^N$ by
\begin{equation}\label{map5}
x \mapsto [g, f](x) = f\bigl(g^{-1} x\bigr).
\end{equation}
If $1$ is the identity element in $G$ then
\begin{equation}\label{map6}
x \mapsto [1, f](x) = f(x)
\end{equation}
is obvious.  And if $g_1$ and $g_2$ belong to $G$ then
\begin{align}\label{map7}
\begin{split}
\bigl[g_1, [g_2, f]\bigr](x) &= \big[ g_{2} , f\big](g_{1}^{-1} x)\\
				      &= f\bigl(g_2^{-1} g_1^{-1} x\bigr) \\
				      &= f\bigl((g_1 g_2)^{-1} x\bigr) \\
				      &= [g_1 g_2, f](x).
\end{split}
\end{align}
It follows from (\ref{map6}) and (\ref{map7}) that (\ref{map5}) defines an action of the group $G$ on the collection of
functions $\Z^N$.

Let $\{s_1, s_2, \dots , s_I\}$ be a transversal for the left cosets of the subgroup $H K$ in $G$, so that 
\begin{equation*}\label{map10}
G = \bigcup_{i = 1}^I s_i H K
\end{equation*} 
is a disjoint union.  Let $\{t_1, t_2, \dots , t_J\}$ be a transversal for the left cosets of the subgroup $K$ in $H K$, so that
\begin{equation}\label{map15}
H K = \bigcup_{j = 1}^J t_j K
\end{equation} 
is a disjoint union.  It follows that
\begin{equation*}\label{map20}
G = \bigcup_{i = 1}^I s_i H K = \bigcup_{i = 1}^I \bigcup_{j = 1}^J s_i t_j K,
\end{equation*} 
and therefore 
\begin{equation}\label{map25}
\{s_i t_j : \text{$i = 1, 2, \dots , I$ and $j = 1, 2, \dots , J$}\}
\end{equation} 
is a transversal for left cosets of the subgroup $K$ in $G$.  Thus a function $f$ in $\Z^N$ is uniquely determined by its 
values on the distinct left coset representatives (\ref{map25}).  We define the subgroup 
\begin{equation}\label{map30}
\eL = \bigg\{f \in \Z^N : \text{$\sum_{j = 1}^J f(s_i t_j) = 0$ for each $i = 1, 2, \dots , I$}\bigg\}.
\end{equation} 
Because $f$ in $\Z^N$ is constant on each left coset of $K$, the identity (\ref{map15}) implies that
\begin{equation}\label{map31}
\sum_{g \in s_i H K} f(g) = |K| \sum_{j = 1}^J f(s_i t_j)
\end{equation}
for each left coset $s_i H K$ in $G$.  Thus $f$ in $\Z^N$ belongs to the subgroup $\eL$ if and only if the
sum of the values that $f$ takes on each coset of $H K$ is zero.  In particular, the choice of transversals
$\{s_1, s_2, \dots , s_I\}$ and $\{t_1, t_2, \dots , t_J\}$ does not effect the definition of the subgroup $\eL$.

As the subsets
\begin{equation*}\label{map35}
\{s_i t_j : \text{$j = 1, 2, \dots , J$}\},\quad\text{where $i = 1, 2, \dots , I$},
\end{equation*} 
are disjoint, the $I$ linear equations satisfied by functions $f$ in $\eL$ are clearly independent.  Hence
the subgroup $\eL$ has rank $N - I$.  We note that
\begin{equation*}\label{map36}
N - I = I(J -1).
\end{equation*} 
We have observed that the group $G$ acts on the group $\Z^N$ by (\ref{map5}).  We now show that the subgroup
$\eL$ is invariant under this action.  That is, $G$ acts on $\eL$.

\begin{lemma}\label{lemmap1}  Let $g$ be an element of $G$, and let $f$ be a function in $\eL$.  Then the function
\begin{equation*}\label{map39}
x \mapsto [g ,f](x) = f\bigl(g^{-1} x\bigr)
\end{equation*}
belongs to $\eL$.
\end{lemma}

\begin{proof}  For each $i = 1, 2, \dots , I$ we have
\begin{equation}\label{map40}
|K| \sum_{j = 1}^J [g, f](s_i t_j) = |K| \sum_{j = 1}^J f\bigl(g^{-1} s_i t_j\bigr) = \sum_{h \in g^{-1} s_i H K} f(h).
\end{equation}
But the sum on the right of (\ref{map40}) is zero because $f$ belongs to $\eL$ and $g^{-1} s_i H K$ is a left
coset of $H K$.
\end{proof}

We now consider and solve the following problem: construct a function $\lambda$ in $\eL$ such that the subgroup
\begin{equation}\label{map45}
\langle [g, \lambda] : g \in G\rangle
\end{equation}
generated by the orbit of $\lambda$ under the action of $G$, has rank $N - I$ in $\eL$.  We define $\lambda$ by
\begin{equation}\label{map49}
\lambda(g) = \begin{cases}  J - 1&   \text{if $g$ belongs to $K$,}\\
                                        \ \ -1 & \text{if $g$ belongs to $H K$, but does not belong to $K$,}\\
                                        \ \ \ 0  & \text{if $g$ belongs to $G$, but does not belong to $H K$.}\end{cases}
\end{equation}
We will prove that for this choice of $\lambda$, the subgroup (\ref{map45}) has rank $N - I$.

If $s_i H K$ is a left coset of $H K$, but not equal to $H K$, then
\begin{equation}\label{map52}
\sum_{g \in s_i H K} \lambda(g) = 0
\end{equation}
is obvious because each term in the sum is zero.  When we sum over the subgroup $H K$ we find that
\begin{equation}\label{map54}
\sum_{g \in H K} \lambda(g) = |K| \sum_{j = 1}^J \lambda(t_j) = |K|\bigl(J - 1 - (J -1)\bigr) = 0.
\end{equation}
It follows that $\lambda$ belongs to $\eL$.  

\begin{lemma}\label{lemmap2}  Let $\lambda$ in $\eL$ be defined by {\rm (\ref{map49})}, and let $\{t_1, t_2, \dots , t_J\}$ be a 
transversal for the left cosets of $K$ in $H K$.  Then each subset of cardinality $J -1$, contained in the collection of functions
\begin{equation*}\label{map58}
\{[t_j, \lambda] : j = 1, 2, \dots , J\},
\end{equation*}
is linearly independent.
\end{lemma}

\begin{proof}  Let $\mu$ be the function $\lambda$ with domain restricted to those left cosets of $K$ that are contained in 
$H K$, so that
\begin{equation}\label{map59}
\mu(h) = \begin{cases}  J - 1&   \text{if $h$ belongs to $K$,}\\
                                     \ \ -1 & \text{if $h$ belongs to $H K$, but does not belong to $K$.}\end{cases}
\end{equation}
For $x$ in $H K$ we have
\begin{equation*}\label{map60}
\mu\bigl(t_j^{-1} x\bigr) = \lambda\bigl(t_j^{-1} x\bigr) = [t_j, \lambda](x).
\end{equation*} 
Therefore it suffices to show that each subset of cardinality $J -1$, contained in the collection of functions
\begin{equation}\label{map61}
\big\{\mu\bigl(t_j^{-1} x\bigr) : j = 1, 2, \dots , J\big\},
\end{equation}
is linearly independent.  Here each function $x \mapsto \mu\bigl(t_j^{-1} x\bigr)$ is defined on the set of $J$
distinct left cosets of $K$ in $H K$, and these cosets are represented by the elements of the transversal
$\{t_1, t_2, \dots  t_J\}$.

Let 
\begin{equation*}\label{map62}
M = \bigl(\mu\bigl(t_j^{-1} t_i\bigr)\bigr)
\end{equation*}
be the $J \times J$ integer matrix, where $i = 1, 2, \dots , J$ indexes rows, and $j = 1, 2, \dots , J$ indexes
columns.  If $i \not= j$ then $t_j^{-1} t_i$ does not belong to $K$, and it follows from (\ref{map59}) that 
\begin{equation*}\label{map63}
\mu\bigl(t_j^{-1} t_i\bigr) = -1 < 0.
\end{equation*}
For each $j = 1, 2, \dots , J$ the elements in the set
\begin{equation*}\label{map64}
\big\{t_j^{-1} t_1, t_j^{-1} t_2, \dots , t_j^{-1} t_J\big\}
\end{equation*}
form a transversal for the left cosets of $K$ in $H K$.  Hence we get
\begin{equation*}\label{map74}
\sum_{i = 1}^J \mu\bigl(t_j^{-1} t_i\bigr) = (J - 1) - \sum_{\substack{i = 1\\i \not= j}}^J 1 = 0.
\end{equation*}
We have shown that $M$ satisfies the hypotheses (\ref{gen201}) and (\ref{gen203}) in the statement of Lemma \ref{lemgen4}.  
Hence by that result the matrix $M$ has $\Q$-rank and $\R$-rank equal to $J-1$.   If $z_1, z_2, \dots , z_J$ are integers,
not all of which are zero, such that
\begin{equation*}\label{map79}
\sum_{j = 1}^J z_j \mu\bigl(t_j^{-1} t_i\bigr) = 0,\quad\text{for each $i = 1, 2, \dots , J$},
\end{equation*}
then (ii) of Lemma \ref{lemgen4} asserts that the integers $z_j$ are all positive or all negative.  In particular, each subset of
cardinality $J - 1$, contained in the collections of functions (\ref{map61}), is linearly independent.
\end{proof}

We are now able to prove that the subgroup (\ref{map45}) has rank $N - I$.  This is contained in the following more
precise result.

\begin{lemma}\label{lemmap3}  Let $\lambda$ in $\eL$ be defined by {\rm (\ref{map49})}, let $\{s_1, s_2, \dots , s_I\}$
be a transversal for the left cosets of $H K$ in $G$, and let $\{t_1, t_2, \dots , t_J\}$ be a transversal for the left cosets
of $K$ in $H K$.  For each $i = 1, 2, \dots , I$, let
\begin{equation*}\label{map90}
\J_i \subseteq \{1, 2, \dots , J\}
\end{equation*}
be a subset of cardinality $|\J_i| = J-1$.  Then the collection of $N - I$ functions
\begin{equation}\label{map93}
\big\{[s_i t_j, \lambda] : \text{$i = 1, 2, \dots , I,$ and $j \in \J_i$}\big\} \subseteq \eL,
\end{equation}
is linearly independent.  Moreover, for each $i = 1, 2, \dots , I$, the $J - 1$ functions in the subcollection
\begin{equation*}\label{map95}
\big\{[s_i t_j, \lambda] : j \in \J_i\big\},
\end{equation*}
are supported on the left coset $s_i H K$.
\end{lemma}

\begin{proof}   It follows from the definition (\ref{map49}) that the function $\lambda$ is supported on the subgroup
$H K$.  Hence for each $i = 1, 2, \dots , I$, the function
\begin{equation*}\label{map97}
x \mapsto [s_i t_j, \lambda](x) = \lambda\bigl((s_i t_j)^{-1} x\bigr)
\end{equation*}
is supported on the left coset 
\begin{equation*}\label{map103}
s_i t_j H K = s_i H K.
\end{equation*}
This verifies the last statement in the lemma.  

Let
\begin{equation*}\label{map107}
\{z(i, j) : \text{$i = 1, 2, \dots , I$ and $j \in \J_i$}\}
\end{equation*}
be a collection of integers, not all of which are zero.  Assume that the function
\begin{equation*}\label{map112}
x \mapsto \sum_{i = 1}^I \sum_{j \in \J_i} z(i, j) \lambda\bigl((s_i t_j)^{-1} x\bigr)
\end{equation*}
is identical zero on $G$.  For each $i = 1, 2, \dots , I$ the function
\begin{equation}\label{map117}
x \mapsto \sum_{j \in \J_i} z(i, j) \lambda\bigl((s_i t_j)^{-1} x\bigr)
\end{equation}
is supported on the left coset $s_i H K$, and the left cosets
\begin{equation*}\label{map121}
s_i H K,\quad\text{where $i = 1, 2, \dots , I$},
\end{equation*}
are obviously disjoint.  Hence each function (\ref{map117}) is identically zero on $G$.

Because $\{t_1, t_2, \dots , t_J\}$ is a transversal for the left cosets of $K$ in $H K$, for $i = 1, 2, \dots , I$,
each left coset of $K$ in $s_i H K$ is represented by a unique element in the set
\begin{equation*}\label{map129}
\{s_i t_1, s_i t_2, \dots , s_i t_J\}.
\end{equation*}
It follows that for each $i = 1, 2, \dots, I$, we have
\begin{equation}\label{map133}
0 = \sum_{j \in \J_i} z(i, j) \lambda\bigl((s_i t_j)^{-1} s_i t_k\bigr) = \sum_{j \in \J_i} z(i, j) \lambda\bigl(t_j^{-1} t_k\bigr)
\end{equation}
for each $k = 1, 2, \dots , J$.  But $\J_i$ has cardinality $J - 1$, and therefore the identity (\ref{map133})
contradicts the statement of Lemma \ref{lemmap2}.  We  conclude that the integers
$z(i, j)$ do not exist.  That is, the functions in the collection (\ref{map93}) are linearly independent.

As the functions in the collection (\ref{map93}) are linearly independent, they are distinct, and the cardinality of the 
collection (\ref{map93}) is
\begin{equation*}\label{map137}
\sum_{i = 1}^I |\J_i| = I(J - 1) = IJ - I = N - I.
\end{equation*}
This proves the lemma.
\end{proof}

\section{The Galois action on places}

We assume that $l$ and $k$ are algebraic number fields such that
\begin{equation*}\label{int0}
\Q \subseteq k \subseteq l.
\end{equation*}
At each place $v$ of $k$ we write $k_v$ for the completion of $k$ at $v$, so that $k_v$ is a local field.  We
select two absolute values $\|\ \|_v$ and $|\ |_v$ from the place $v$.  The absolute value $\|\ \|_v$ extends the usual archimedean 
or non-archimedean absolute value on the subfield $\Q$.  Then $|\ |_v$ must be a power of $\|\ \|_v$, and we set
\begin{equation}\label{int1}
|\ |_v = \|\ \|_v^{d_v/d},
\end{equation}
where $d_v = [k_v : \Q_v]$ is the local degree of the extension, and $d = [k : \Q]$ is the global degree.  In a similar manner
we write $w$ for a place of $l$, $l_w$ for the completion of $l$ at $w$, and we normalize two absolute values $\|\ \|_w$
and $|\ |_w$ from the place $w$ in a similar manner.  We write $w|v$ when $\|\ \|_w$ extends the absolute value
$\|\ \|_v$ from $k$ to $l$.  Then we write $W_v(l/k)$ for the finite set of all places $w$ of $l$ such that $w|v$.

With these normalizations the height of an algebraic number $\alpha \not= 0$ that belongs to $l$ is given by
\begin{equation}\label{int3}
h(\alpha) = \sum_w \log^+ |\alpha|_w = \hh \sum_w \bigl|\log |\alpha|_w\bigr|.
\end{equation}
Each sum in (\ref{int3}) is over the set of all places $w$ of $l$, and the equality between the two sums follows 
from the product formula.  Then $h(\alpha)$ depends on the algebraic number $\alpha \not= 0$, but it does not depend on 
the number field $l$ that contains $\alpha$.  We have already noted in (\ref{sum5}) that the height is well defined as a map
\begin{equation*}\label{int4}
h : \G_l \rightarrow [0, \infty).
\end{equation*}

If $l/k$ is a finite, Galois extension then the Galois group $\Aut(l/k)$ acts transitively on the set $W_v(l/k)$ of places $w$ of $l$ 
that lie above a fixed place $v$ of $k$ (see Tate \cite{JTT}).  If $\sigma$ is 
an element of $\Aut(l/k)$ and $w$ is a place of $l$, then $\sigma w$ is the unique place of $l$ that satisfies the identity
\begin{equation}\label{tate-1}
\|\sigma^{-1}(\gamma)\|_w = \|\gamma\|_{\sigma w}
\end{equation} 
for each $\gamma$ in $l$.  Because $\sigma$ fixes elements of $k$, we find that the restriction of (\ref{tate-1}) to
$k$ is equal to the restriction of $\|\ \|_w$ to $k$.  That is, $\sigma w$ and $w$ are both places in the set $W_v(l/k)$.
For a Galois extension all local degrees over a fixed place of $k$ are equal.  Alternatively,
the map $w \mapsto [l_w : k_v]$ is constant on places $w$ that belong to $W_v(l/k)$.  This observation easily implies that
\begin{equation}\label{tate0}
|\sigma^{-1}(\gamma)|_w = |\gamma|_{\sigma w}
\end{equation} 
for each $\gamma$ in $l$.  

Now assume that $l/\Q$ is a finite, Galois extension and $k$ is an intermediate field.  We write 
\begin{equation*}\label{tate2}
G = \Aut(l/\Q),\quad\text{and}\quad H = \Aut(l/k),
\end{equation*}
so that $H \subseteq G$ is a subgroup of $G$, and $H$ is the group of automorphisms attached to the Galois extension $l/k$. 
If $w$ is a place of $l$, if $v$ is a place of $k$ such that $w|v$, if $l_w$ is the completion of $l$ at $w$, and $k_v$ is the completion 
of $k$ at $v$, then $l_w/k_v$ is a Galois extension.  It can be shown (see Tate \cite{JTT})
that the Galois group $\Aut(l_w/k_v)$ is isomorphic to the stabilizer
\begin{equation*}\label{tate3}
H_w = \{\sigma \in H : \sigma w = w\}.
\end{equation*} 
As the completion of an archimedean local field is either $\R$ or $\C$, it follows that each stabilizer $H_w$ is either trivial,
or is cyclic of order $2$.  More precisely, we have
\begin{align*}\label{tate4}
\begin{split}
|H_w| &= 1\quad\text{if and only if either $l_w \cong k_v \cong \R$ for all $w$ with $w|v$,}\\
	  &\qquad\qquad\qquad\text{or $l_w \cong k_v \cong \C$ for all $w$ with $w|v$,}
\end{split}
\end{align*}
and
\begin{equation*}\label{tate5}
|H_w| = 2\quad\text{if and only if both $k_v \cong \R$ and $l_w \cong \C$ for all $w$ with $w|v$.}
\end{equation*}
Of course the same remark applies to the stabilizer
\begin{equation*}\label{tate6}
G_w = \{\sigma \in G : \sigma w = w\},
\end{equation*}
but now $\Q$ has one archimedean place, and $\Q_{\infty} \cong \R$.  We find that
\begin{equation}\label{tate7}
|G_w| = 1\quad\text{if and only if $l_w \cong \R$ for all $w$ with $w|\infty$,}
\end{equation}
and
\begin{equation}\label{tate8}
|G_w| = 2\quad\text{if and only if $l_w \cong \C$ for all $w$ with $w|\infty$.}
\end{equation}
Clearly (\ref{tate7}) occurs when $l/\Q$ is a totally real Galois extension, and (\ref{tate8}) occurs when $l/\Q$ is
a totally complex Galois extension.

Let $\widehat{w}$ be a particular archimedean place of $l$.  As before we write
\begin{equation*}\label{tate30}
G_{\widehat{w}} = \{\sigma \in G : \sigma \widehat{w} = \widehat{w}\},
\end{equation*}
for the stabilizer of $\widehat{w}$.  We have $|G_{\widehat{w}}| = [l_{\widehat{w}} : \Q_{\widehat{w}}]$, and therefore
\begin{equation*}\label{tate32}
|G_{\widehat{w}}| = \begin{cases} 1&\text{if $l/\Q$ is totally real,}\\
                                                      2&\text{if $l/\Q$ is totally complex.}\end{cases}
\end{equation*}
Write $[G : G_{\widehat{w}}] = N$, so that $N$ is the number of archimedean places of $l$.
Let $\tau_1, \tau_2, \dots , \tau_N$ be a complete 
set of distinct representatives for the left cosets of the subgroup $G_{\widehat{w}}$.  Then
\begin{equation*}\label{tate39}
\big\{\tau_n \widehat{w} : n = 1, 2, \dots , N\big\} = W_{\infty}(l/\Q)
\end{equation*}
is a complete set of distinct archimedean places of $l$.  To verify this, observe that
if $\tau_m \widehat{w} = \tau_n \widehat{w}$, then $\tau_m^{-1} \tau_n$ belongs to $G_{\widehat{w}}$.
and it follows that $\tau_n$ is an element of the coset $\tau_m G_{\widehat{w}}$.  This is impossible because 
$\tau_1, \tau_2, \dots , \tau_N$ is a complete set of distinct representatives for the left cosets of $G_{\widehat{w}}$.
Therefore we have
\begin{equation}\label{tate41}
\big\{|\ |_{\tau_n \widehat{w}} : n = 1, 2, \dots , N\big\} = \big\{|\ |_w : w \in W_{\infty}(l/\Q)\big\}.
\end{equation}

\section{Special Minkowski units}\label{SMU}

In this section we assume that $l/\Q$ is a finite Galois extension, and we write $F_l$ for the free group (\ref{first10}) of positive rank.
We say that $\beta$ in $F_l$ is a {\it special Minkowski unit} if there exists an archimedean place $\widehat{w}$ of $l$ such that
\begin{equation}\label{intro1}
\log |\beta|_w < 0\quad\text{for all archimedean places $w$ such that $w \not= \widehat{w}$}. 
\end{equation}
If $\beta$ is a special Minkowski unit with respect to $\widehat{w}$, then by the product formula we have
\begin{equation*}\label{intro2}
0 = \sum_{w | \infty} \log |\beta|_w = \log |\beta|_{\widehat{w}} + \sum_{\substack{w | \infty\\ w \not= \widehat{w}}} \log |\beta|_w,
\end{equation*}
and therefore
\begin{equation*}\label{intro3}
0 < \log |\beta|_{\widehat{w}}.
\end{equation*}
If $\beta_1$ and $\beta_2$ are both special Minkowski units with respect to $\widehat{w}$, then it is trivial that
the product $\beta_1 \beta_2$ is also a special Minkowski unit with respect to $\widehat{w}$.  More generally, for
each archimedean place $\widehat{w}$ we define
\begin{equation*}\label{intro4}
\M_{\widehat{w}} = \{\alpha \in F_l : \text{$\log |\alpha|_w < 0$ for all $w | \infty$ such that $w \not= \widehat{w}$}\},
\end{equation*}
so that $\M_{\widehat{w}}$ is the set of all special Minkowski units with respect to $\widehat{w}$.  Then each subset  
$\M_{\widehat{w}}$ is clearly a multiplicative semi-group in $F_l$.  Later we will show that if $\aA \subseteq F_l$ is
a subgroup of maximal rank, then
\begin{equation}\label{intro5}
\aA \cap \M_{\widehat{w}}
\end{equation}
is not empty.

\begin{lemma}\label{lemspecial1}  Let $\beta$ be an element of $F_l$.  Then the following conditions are equivalent.
\begin{itemize}
\item[(i)]  The element $\beta$ is a special Minkowski unit with respect to the archimedean place $\widehat{w}$.
\item[(ii)]  If $T = \{\tau_1, \tau_2, \dots , \tau_N\}$ is a transversal for the left cosets of the subgroup $G_{\widehat{w}}$, then
\begin{equation}\label{intro7}
\log \bigl| \tau_m^{-1} \tau_n (\beta)\bigr|_{\widehat{w}} < 0,\quad\text{whenever $m \not= n$}.
\end{equation}
\end{itemize}
\end{lemma}

\begin{proof}  Assume that (i) holds.  If $m \not= n$ we have 
$\tau_m G_{\widehat{w}} \not= \tau_n G_{\widehat{w}}$, and therefore
\begin{equation*}\label{intro9}
w = \bigl(\tau_m^{-1} \tau_n\bigr)^{-1} \widehat{w} = \tau_n^{-1} \tau_m \widehat{w} \not= \widehat{w}.
\end{equation*}
Then it follows from (\ref{intro1}) that
\begin{equation*}\label{intro11}
\log \bigl| \tau_m^{-1} \tau_n (\beta)\bigr|_{\widehat{w}} = \log |\beta|_w < 0,
\end{equation*}
which verifies (ii).

Assume that (ii) holds.  We recall that $G$ acts transitively on the collection $W_{\infty}(l/\Q)$ of archimedean places of $l$.  If $\eta$ in
$G$ satisfies $\eta \widehat{w} = w$, then we have
\begin{equation*}\label{intro13}
\{\sigma \in G : \sigma \widehat{w} = w\} = \eta G_{\widehat{w}}.
\end{equation*}
As $\{\tau_1, \tau_2, \dots , \tau_N\}$ is a transversal for the left cosets of the subgroup $G_{\widehat{w}}$, there exists a 
a pair $(m, n)$ such that 
\begin{equation*}\label{intro15}
\bigl(\tau_m^{-1} \tau_n\bigr)^{-1} G_{\widehat{w}} = \eta G_{\widehat{w}}.
\end{equation*}
Therefore we have
\begin{equation*}\label{intro17}
\bigl(\tau_m^{-1} \tau_n\bigr)^{-1} \widehat{w} = w.
\end{equation*}
If $m \not= n$, then $w \not= \widehat{w}$, and (\ref{intro7}) implies that
\begin{equation*}\label{intro21}
\log |\beta|_w = \log \bigl| \tau_m^{-1} \tau_n (\beta)\bigr|_{\widehat{w}}  < 0.
\end{equation*}
It follows that $\beta$ is a special Minkowski unit with respect to $\widehat{w}$.
\end{proof}

The situation is further clarified by the following basic result.

\begin{lemma}\label{lemspecial2}  Let $\beta$ in $F_l$ be a special Minkowski unit with respect to the 
archimedean place $\widehat{w}$.
\begin{itemize}
\item[(i)]  If $T = \{\tau_1, \tau_2, \dots , \tau_N\}$ is a transversal for the left cosets of the subgroup 
$G_{\widehat{w}}$, then each subset of
\begin{equation}\label{special8}
\{\tau_1(\beta), \tau_2(\beta), \dots , \tau_N(\beta)\}
\end{equation}
with cardinality $N-1$, is multiplicatively independent in $F_l$.
\item[(ii)]  The number $\beta$ is a Minkowski unit.
\end{itemize}
\end{lemma}

\begin{proof}  Define the $N \times N$ real matrix
\begin{equation}\label{special9}
M(\beta, T, \widehat{w}) = \bigl(\log |\tau_m^{-1} \tau_n(\beta)|_{\widehat{w}}\bigr),
\end{equation}
where $m = 1, 2, \dots , N$ indexes rows, and $n = 1, 2, \dots , N$ indexes columns.  Because $T$ is a transversal
for the left cosets of $G_{\widehat{w}}$, we have
\begin{equation*}\label{special11}
\big\{\tau_m \widehat{w} : m = 1, 2, \dots N\big\} = W_{\infty}(l/\Q).
\end{equation*}
Therefore the matrix $M(\beta, T, \widehat{w})$ satisfies
\begin{align}\label{special13}
\begin{split}
\sum_{m = 1}^N \log |\tau_m^{-1} \tau_n(\beta)|_{\widehat{w}} 
	&= \sum_{m = 1}^N \log |\tau_n(\beta)|_{\tau_m \widehat{w}}\\
	&= \sum_{w | \infty} \log |\tau_n(\beta)|_w\\
	&= 0
\end{split}
\end{align}
by the product formula.  If $m \not= n$, then it follows from Lemma \ref{lemspecial2} that
\begin{equation}\label{special17}
\log  |\tau_m^{-1} \tau_n(\beta)|_{\widehat{w}} < 0.
\end{equation}
The identity (\ref{special13}) and the inequality (\ref{special17}) verify the hypotheses (\ref{gen201}) and (\ref{gen203}) in the 
statement of Lemma \ref{lemgen4}.  We conclude that the $\Q$-rank and the $\R$-rank of $M(\beta, T, \widehat{w})$ are both 
equal to $N-1$, and each $(N - 1) \times (N -1)$ submatrix of $M(\beta, T, \widehat{w})$ is nonsingular.  The set of columns of 
the matrix $M(\beta, T, \widehat{w})$ is
\begin{equation*}\label{special19}
\big\{\bigl(\log |\tau_m^{-1} \tau_n(\beta)|_{\widehat{w}}\bigr) : n = 1, 2, \dots , N\big\}
	= \big\{\bigl(\log |\tau_n(\beta)|_w\bigr) : n = 1, 2, \dots , N\big\},
\end{equation*}
and we conclude that each subset of $N-1$ distinct columns is $\Q$-linearly independent. This clearly implies the conclusion
(i) in the statement of the lemma, and (ii) follows immediately.
\end{proof}

If the Galois extension $l/\Q$ is totally real, then the subgroup $G_{\widehat{w}}$ is trivial, and the transversal $T$
that occurs in the proof of Lemma \ref{lemspecial1} is the group $G$.  In this case we have the identity (\ref{special13})
for column sums, and the corresponding identity 
\begin{align}\label{special25}
\begin{split}
\sum_{n = 1}^N \log |\tau_m^{-1} \tau_n(\beta)|_{\widehat{w}} 
	&= \sum_{n = 1}^N \log |\tau_n(\beta)|_{\tau_m \widehat{w}}\\
	&= \log\bigl| \norm_{l/\Q}(\beta)\big|_{\tau_m \widehat{w}}\\
	& = 0
\end{split}	
\end{align}
for row sums.  It will be useful to work in an analogous situation when $l/\Q$ is a totally complex Galois extension.

\begin{lemma}\label{lemspecial3}  Let $l/\Q$ be a totally complex Galois extension, $\widehat{w}$ an archimedean
place of $l$, and let $\beta$ be a special Minkowski unit with respect to $\widehat{w}$.  Write
\begin{equation}\label{special31}
G_{\widehat{w}} = \{1, \rho\},\quad\text{where $\rho^2 = 1$}.
\end{equation}
Then both $\rho(\beta)$ and $\beta \rho(\beta)$ are special Minkowski units with respect to $\widehat{w}$.  Moreover, if
$T = \{\tau_1, \tau_2, \dots , \tau_N\}$ is a transversal for the left cosets of the subgroup $G_{\widehat{w}}$, then we have
\begin{equation}\label{special33}
\sum_{n = 1}^N \log|\tau_m^{-1} \tau_n(\beta \rho(\beta))|_{\widehat{w}} = 0
\end{equation}
for each $m = 1, 2, \dots , N$.
\end{lemma}

\begin{proof}  We define the $N \times N$ real matrix
\begin{equation*}\label{special39}
M(\beta, T, \widehat{w}) = \bigl(\log |\tau_m^{-1} \tau_n(\beta)|_{\widehat{w}}\bigr),
\end{equation*}
where $m = 1, 2, \dots , N$ indexes rows, and $n = 1, 2, \dots , N$ indexes columns.  As in our proof of Lemma \ref{lemspecial2}, 
the matrix $M(\beta, T, \widehat{w})$ satisfies (\ref{special13}) and (\ref{special17}).  Using (\ref{special31}) we find that
\begin{equation*}\label{special41}
T \rho = \{\tau_1 \rho, \tau_2 \rho, \dots , \tau_N \rho\}
\end{equation*}
is a second transversal for the left cosets of $G_{\widehat{w}}$.  Hence the matrix
\begin{equation*}\label{special42}
M(\beta, T \rho, \widehat{w}) = \bigl(\log \bigl|(\tau_m \rho)^{-1} \tau_n \rho(\beta)\bigr|_{\widehat{w}}\bigr),
\end{equation*}
where $m = 1, 2, \dots , N$ indexes rows and $n = 1, 2, \dots , N$ indexes columns, also satisfies the identity
\begin{align}\label{special43}
\begin{split}
\sum_{m = 1}^N \log \bigl|(\tau_m \rho)^{-1} \tau_n \rho(\beta)\bigr|_{\widehat{w}}
	&= \sum_{m = 1}^N \log |\tau_n \rho(\beta)|_{\tau_m \rho \widehat{w}}\\
	&= \sum_{w | \infty} \log |\tau_n \rho(\beta)|_w\\
	&= 0,
\end{split}
\end{align}
and the inequality
\begin{equation*}\label{special47}
\log \bigl|(\tau_m \rho)^{-1} \tau_n \rho(\beta)\bigr|_{\widehat{w}} < 0.
\end{equation*}
Because $\rho$ belongs to the stabilizer $G_{\widehat{w}}$, we have $\rho \widehat{w} = \widehat{w}$.  Therefore
the $(m, n)$ entry in the matrix $M(\beta, T \rho, \widehat{w})$ is
\begin{align}\label{special53}
\begin{split}
\log \bigl|(\tau_m \rho)^{-1} \tau_n \rho(\beta)\bigr|_{\widehat{w}}
	&= \log |\tau_n \rho(\beta)|_{\tau_m \rho \widehat{w}}\\
	&= \log |\tau_n \rho(\beta)|_{\tau_m \widehat{w}}\\
	&= \log \bigl|\tau_m^{-1} \tau_n (\rho(\beta))\bigr|_{\widehat{w}}.
\end{split}
\end{align}
The identity (\ref{special53}) implies that
\begin{equation*}\label{special55}
M(\beta, T \rho, \widehat{w}) = M(\rho(\beta), T, \widehat{w}).
\end{equation*}
As $T$ is an arbitrary left transversal for $G_{\widehat{w}}$, it follows from Lemma \ref{lemspecial2}
that $\rho(\beta)$ is a special Minkowski unit with respect to the place $\widehat{w}$.  We have shown that both
$\beta$ and $\rho(\beta)$ belong to the semigroup $\M_{\widehat{w}}$ of special Minkowski units.  Therefore 
the product $\beta \rho(\beta)$ is a special Minkowski unit.

Because
\begin{equation*}\label{special58}
G = \{\tau_1, \tau_2, \dots , \tau_N\} \cup \{\tau_1 \rho, \tau_2 \rho, \dots , \tau_N \rho\} = T \cup T \rho,
\end{equation*}
we find that
\begin{align*}\label{special63}
\begin{split}
\sum_{n = 1}^N \log|\tau_m^{-1} \tau_n(\beta \rho(\beta)|_{\widehat{w}} 
	&= \sum_{n = 1}^N \log |\tau_n(\beta \rho(\beta)|_{\tau_m \widehat{w}}\\
	&= \sum_{n = 1}^N \log |\tau_n(\beta)|_{\tau_m \widehat{w}} 
		+ \sum_{n = 1}^N \log|\tau_n \rho(\beta)|_{\tau_m \widehat{w}}\\
	&= \log |\norm_{l/\Q}(\beta)|_{\tau_m \widehat{w}}\\
	&= 0.
\end{split}
\end{align*}
This proves (\ref{special33}).
\end{proof}

Let $\aA \subseteq F_l$ be a subgroup with maximal rank,
and let $\widehat{w}$ be an archimedean place of $l$.  We now prove that the subgroup $\aA$ contains a special Minkowski
unit $\beta$ with respect to the place $\widehat{w}$.  It follows from Lemma \ref{lemspecial2} that $\beta$ is a Minkowski unit in 
$\aA$.  We construct $\beta$ so that the Weil height of $\beta$ is comparable with the sum of the 
heights of a basis for the subgroup $\aA$.  We also give a bound on the index of the subgroup generated by the conjugates 
of $\beta$ in the full group of units $F_l$.

\begin{theorem}\label{thmspecial1}  Let $l/\Q$ be a Galois extension with $N$ archimedean places, and let $\widehat{w}$ 
be a particular archimedean place of $l$.  Let $\eta_1, \eta_2, \dots , \eta_{N-1}$ be multiplicatively independent units 
in $F_l$, and write
\begin{equation*}\label{new108}
\aA = \langle \eta_1, \eta_2, \dots , \eta_{N-1}\rangle \subseteq F_l
\end{equation*}
for the subgroup of rank $N - 1$ that they generate.  Then there exists 
a special Minkowski unit $\beta$ with respect to $\widehat{w}$ that satisfies the following conditions.
\begin{itemize}
\item[(i)]  The unit $\beta$ belongs to $\aA$.
\item[(ii)]  The height of $\beta$ is bounded by
\begin{equation}\label{new110}
h(\beta) \le 2 \sum_{n=1}^{N-1} h(\eta_n).
\end{equation}
\item[(iii)]  If $T = \{\tau_1, \tau_2, \dots , \tau_N\}$ is a transversal for the left cosets of the subgroup 
$G_{\widehat{w}}$, then the $N \times N$ matrix
\begin{equation}\label{new112}
M(\beta, T, \widehat{w}) = \bigl(\log |\tau_m^{-1} \tau_n(\beta)|_{\widehat{w}}\bigr),
\end{equation}
where $m = 1, 2, \dots , N$ indexes rows and $n = 1, 2, \dots , N$ indexes columns, has $\Q$-rank and $\R$-rank equal to $N-1$.
\item[(iv)]  Each $(N-1) \times (N-1)$ submatrix of $M(\beta, T, \widehat{w})$ is nonsingular.
\item[(v)]  If $\bwy \not= \bo$ is a point in $\R^N$ such that
\begin{equation*}\label{new114}
0 = \sum_{n = 1}^N y_n \log |\tau_n(\beta)|_w,\quad\text{for each place $w$ in $W_{\infty}(l/\Q)$},
\end{equation*}
then the co-ordinates $y_1, y_2, \dots , y_N$ are all positive, or all negative.
\item[(vi)]  The subgroup
\begin{equation*}\label{new118}
\bB = \langle \tau_1(\beta), \tau_2(\beta), \dots , \tau_N(\beta)\rangle \subseteq F_l,
\end{equation*}
generated by the conjugate units has rank $N - 1$, and index bounded by
\begin{equation}\label{new120}
\Reg(l) [F_l : \bB] \le \bigl([l : \Q] h(\beta)\bigr)^{N-1},
\end{equation}
where $\Reg(l)$ is the regulator of $l$.
\end{itemize}
\end{theorem}

\begin{proof}  Let 
\begin{equation*}
A = \bigl(\log |\eta_n|_w\bigr)
\end{equation*}
be the $(N-1) \times (N-1)$ real matrix, where $w|\infty$ with $w \not= \widehat{w}$ indexes rows, and 
$n = 1, 2, \dots , N-1$ indexes columns.  By Lemma \ref{lemgen1} there exists a point $\bxi$ in $\Z^{N-1}$ such that
\begin{equation}\label{new121}
0 < \sum_{n = 1}^{N-1} \xi_n \log |\eta_n|_w \le \sum_{n = 1}^{N-1} \bigl|\log |\eta_n|_w\bigr|
\end{equation}
for each archimedean place $w$ of $l$ with $w \not= \widehat{w}$.  Then it is obvious that $\bxi \not= \bo$.  Let
\begin{equation*}\label{new123}
\beta^{-1} = \prod_{n = 1}^{N-1} \eta_n^{\xi_n},
\end{equation*}
so that $\beta \not= 1$, and $\beta$ is an element of the subgroup $\aA$.  In view of (\ref{new121}) we have
\begin{equation}\label{new125}
- \sum_{n = 1}^{N-1} \bigl|\log |\eta_n|_w\bigr| \le - \sum_{n = 1}^{N-1} \xi_n \log |\eta_n|_w = \log |\beta|_w < 0
\end{equation}
at each archimedean place $w$ of $l$ with $w \not= \widehat{w}$.  From the product formula we get
\begin{align}\label{new127}
\begin{split}
0 < \log |\beta|_{\widehat{w}} = - \sum_{\substack{w|\infty\\w \not= \widehat{w}}} \log |\beta|_{w}
                                   \le \sum_{n = 1}^{N-1} \sum_{\substack{w|\infty\\w \not= \widehat{w}}} \bigl|\log|\eta_n|_w\bigr|.
\end{split}                                   
\end{align}
This leads to the estimate
\begin{align}\label{new128}
\begin{split}
2 h(\beta) &= \bigl|\log |\beta|_{\widehat{w}}\bigr| + \sum_{\substack{w|\infty\\w \not= \widehat{w}}} \bigl|\log |\beta|_{w}\bigr|\\
		&= \log |\beta|_{\widehat{w}} - \sum_{\substack{w|\infty\\w \not= \widehat{w}}} \log |\beta|_{w}\\
		&\le 2 \sum_{n = 1}^{N-1} \sum_{\substack{w|\infty\\w \not= \widehat{w}}} \bigl|\log|\eta_n|_w\bigr|\\
		&\le 4 \sum_{n=1}^{N-1} h(\eta_n),
\end{split}
\end{align}
which verifies (\ref{new110}).

Because of the identity (\ref{tate0}) we have
\begin{equation}\label{new129}
M(\beta, T, \widehat{w}) = \bigl(\log |\tau_n(\beta)|_{\tau_m \widehat{w}}\bigr) 
						= \bigl(\log \bigl|\tau_m^{-1}\bigl(\tau_n(\beta)\bigr)\bigr|_{\widehat{w}}\bigr),
\end{equation}
where $m = 1, 2, \dots , N$ indexes rows and $n = 1, 2, \dots , N$ indexes columns.  The identity (\ref{new129})
determines an ordering for the archimedean places $w$ that index the rows of $M(\beta, T, \widehat{w})$.  But the 
choice of ordering does not effect the rank of $M(\beta, T, \widehat{w})$.  If $m \not= n$ then 
\begin{equation*}\label{new130}
\bigl(\tau_m^{-1} \tau_n\bigr)^{-1} = \tau_n^{-1} \tau_m
\end{equation*}
is {\it  not} in the subgroup $G_{\widehat{w}}$ that fixes $\widehat{w}$.  It follows from (\ref{new125}) that
\begin{equation}\label{new131}
\log \bigl|\tau_m^{-1}\bigl(\tau_n(\beta)\bigr)\bigr|_{\widehat{w}} < 0
\end{equation}
whenever $m \not= n$.  We also get
\begin{align}\label{new133}
\begin{split}
\sum_{m = 1}^N \log \bigl|\tau_m^{-1}\bigl(\tau_n(\beta)\bigr)\bigr|_{\widehat{w}} 
			    &= \sum_{m = 1}^N \log |\tau_n(\beta)|_{\tau_m\widehat{w}}\\
			    &= \sum_{w|\infty} \log |\tau_n(\beta)|_w\\
			    &= 0,
\end{split}
\end{align}
by appealing to (\ref{tate41}) and the product formula.  It follows from (\ref{new131}) and (\ref{new133}) that
the matrix $M(\beta, T, \widehat{w})$ satisfies the hypotheses of Lemma \ref{lemgen4}.  Hence 
$M(\beta, T, \widehat{w})$ also satisfies the conclusions of Lemma \ref{lemgen4}, and this verifies (iii), (iv), and (v).

We have shown that the column vectors
\begin{equation*}\label{new144}
\big\{\bigl(\log |\tau_n(\beta)|_w\bigr) : n = 1, 2, \dots , N\big\} \subseteq \R^N
\end{equation*}
generate a subgroup of rank $N-1$.   Hence the subgroup
\begin{equation*}\label{new149}
\bB = \langle \tau_n(\beta) : n = 1, 2, \dots , N\rangle \subseteq F_l,
\end{equation*}
generated by their inverse image in $F_l$, also has rank $N-1$.  Then the inequality (\ref{new120}) follows from 
\cite[Theorem 1.1]{akhtari2015}, and the observation that the map
\begin{equation*}\label{new153}
n \mapsto h\bigl(\tau_n(\beta)\bigr)
\end{equation*}
is constant.
\end{proof}

If $N$ is the number of archimedean places of the Galois extension $l/\Q$, then the rank of the group $F_l$ is
$r(l) = N - 1$.  Hence Theorem \ref{maintheorem1} follows immediately from (i), (ii), and (vi), in the statement 
of Theorem \ref{thmspecial1}.  

\section{Relative units I: Definitions}\label{RUI}

Throughout this section we suppose that  $k$ and $l$ are algebraic number fields with 
\begin{equation*}\label{unit2}
\Q \subseteq k \subseteq l.  
\end{equation*}
We write $r(k)$ for the rank of the unit group $O_k^{\times}$, and $r(l)$ for the rank of the unit group $O_l^{\times}$.  
Then $k$ has $r(k) + 1$ archimedean places, and $l$ has $r(l) + 1$ archimedean places.  In general we have 
$r(k) \le r(l)$, and we recall (see \cite[Proposition 3.20]{narkiewicz2010}) that $r(k) = r(l)$ if and only if $l$ is a $\CM$-field, 
and $k$ is the maximal totally real subfield of $l$.  

The norm is a homomorphism of multiplicative groups
\begin{equation*}\label{unit2.7}
\Norm_{l/k} : l^{\times} \rightarrow k^{\times}.
\end{equation*}
If $v$ is a place of $k$, then each element $\alpha$ in $l^{\times}$ satisfies the identity
\begin{equation}\label{unit4}
[l : k] \sum_{w|v} \log |\alpha|_w = \log |\Norm_{l/k}(\alpha)|_v,
\end{equation}
where the absolute values $|\ |_v$ and $|\ |_w$ are normalized as in (\ref{int1}).
It follows from (\ref{unit4}) that the norm, restricted to the subgroup $O_l^{\times}$ of units, is a homomorphism
\begin{equation*}\label{unit6}
\Norm_{l/k} : O_l^{\times} \rightarrow O_k^{\times},
\end{equation*}
and the norm, restricted to the torsion subgroup in $O_l^{\times}$, is also a homomorphism
\begin{equation*}\label{unit8}
\Norm_{l/k} : \Tor\bigl(O_l^{\times}\bigr) \rightarrow \Tor\bigl(O_k^{\times}\bigr).
\end{equation*}
Therefore we get a well defined homomorphism, which we write as
\begin{equation*}\label{unit10}
\norm_{l/k} : O_l^{\times}/\Tor\bigl(O_l^{\times}\bigr) \rightarrow O_k^{\times}/\Tor\bigl(O_k^{\times}\bigr),
\end{equation*}
and define by
\begin{equation}\label{unit11}
\norm_{l/k}\bigl(\alpha \Tor\bigl(O_l^{\times}\bigr)\bigr) = \Norm_{l/k}(\alpha) \Tor\bigl(O_k^{\times}\bigr).
\end{equation}
However, to simplify notation we write
\begin{equation}\label{unit12}
F_k = O_k^{\times}/\Tor\bigl(O_k^{\times}\bigr),\quad\text{and}\quad F_l = O_l^{\times}/\Tor\bigl(O_l^{\times}\bigr),
\end{equation}
so that
\begin{equation}\label{unit13}
\norm_{l/k} : F_l \rightarrow F_k.
\end{equation}
We also write the elements of the quotient groups $F_k$ and $F_l$ as coset representatives, rather than as cosets.
And by abuse of language, we continue to refer to the elements of $F_k$ and $F_l$ as units. 
Obviously $F_k$ and $F_l$ are free abelian groups of rank $r(k)$ and $r(l)$, respectively.  As
\begin{equation*}\label{unit17}
O_k^{\times} \subseteq O_l^{\times},
\end{equation*}  
we can identify $F_k$ with the subgroup
\begin{equation*}\label{unit19}
O_k^{\times}/\Tor\bigl(O_l^{\times}\bigr),
\end{equation*}
and in this way regard $F_k$ as a subgroup of $F_l$.  We also note that (\ref{unit4}) and (\ref{unit13}) imply that
\begin{equation}\label{exunit0}
[l : k] \sum_{w | v} \log |\alpha|_w = \log |\norm_{l/k}(\alpha)|_v
\end{equation}
for each place $v$ of $k$ and each point $\alpha$ in $F_l$.

Following Costa and Friedman \cite{costa1991} and \cite{costa1993}, the subgroup of relative units in $O_l^{\times}$ is defined by
\begin{equation*}\label{unit20}
\big\{\alpha \in O_l^{\times} : \Norm_{l/k}(\alpha) \in \Tor\bigl(O_k^{\times}\bigr)\big\}.
\end{equation*}
Here we work in the free group $F_l$ where the image of the subgroup of relative units is the kernel of the homomorphism 
$\norm_{l/k}$.  Therefore we define the subgroup of {\it relative units} in $F_l$ to be the subgroup
\begin{equation}\label{unit21}
E_{l/k} = \big\{\alpha \in F_l : \norm_{l/k}(\alpha) = 1\big\}.
\end{equation}
We also write
\begin{equation}\label{unit24}
I_{l/k} = \big\{\norm_{l/k}(\alpha) : \alpha \in F_l\big\} \subseteq F_k
\end{equation}
for the image of the homomorphism $\norm_{l/k}$.  If $\beta$ in $F_l$ represents a coset in the subgroup $F_k$, then 
we have
\begin{equation*}\label{unit26.7}
\norm_{l/k}(\beta) = \beta^{[l : k]}.
\end{equation*}  
Therefore the image $I_{l/k} \subseteq F_k$ is a subgroup of rank $r(k)$, and the index satisfies
\begin{equation}\label{unit27}
[F_k : I_{l/k}] < \infty.
\end{equation}
It follows that $E_{l/k} \subseteq F_l$ is a subgroup of rank $r(l/k) = r(l) - r(k)$.  We restrict our attention to extensions 
$l/k$ such that 
\begin{equation}\label{unit29}
1 \le r(k) < r(l).
\end{equation}
The inequality $1 \le r(k)$ implies that $k$ is not $\Q$, and $k$ is not an imaginary, quadratic extension of $\Q$.  
And the inequality $r(k) < r(l)$ implies that $l$ is not a $\CM$-field such that $k$ is the maximal totally real subfield 
of $l$.  Alternatively, the hypothesis (\ref{unit29}) implies that $E_{l/k}$ is a proper subgroup of $F_l$.  As $F_k$ is 
a free group, it follows that the kernel $E_{l/k}$ of the homomorphism (\ref{unit13}) is a direct sum in $F_l$.

\section{Relative Units II: $H$ is normal in $G$}\label{RUII}

We continue to assume that
\begin{equation}\label{relunit501}
\Q \subseteq k \subseteq l,
\end{equation}
that these fields satisfy the inequality (\ref{unit29}), and we also assume that both $l/\Q$ and $k/\Q$ are finite, Galois extensions.  
We write
\begin{equation*}\label{relunit503}
H = \Aut(l/k),\quad\text{and}\quad G = \Aut(l/\Q),
\end{equation*}
but now $H$ is a normal subgroup of $G$.  Therefore we have the canonical homomorphism
\begin{equation}\label{relunit504}
\vphi : G \rightarrow G/H,
\end{equation}
and the isomorphism
\begin{equation}\label{relunit505}
G/H \cong \Aut(k/\Q).
\end{equation}
The isomorphism (\ref{relunit505}) is the map that restricts the domain of a coset representative in $G/H$ to
the subfield $k$.  

\begin{lemma}\label{lemfour1}  Assume that the number fields {\rm (\ref{relunit501})} are such that both $l/\Q$ and 
$k/\Q$ are finite, Galois extensions.  Then for each place $u$ of $\Q$ the map
\begin{equation}\label{relunit507}
v \mapsto \bigl|W_v(l/k)\bigr|
\end{equation}
is constant on the collection of places $v$ in $W_u(k/\Q)$, and the group $G$ acts transitively on the collection of disjoint subsets
\begin{equation}\label{relunit509}
\big\{W_v(l/k) : v \in W_u(k/\Q)\big\}.
\end{equation}
Moreover, for each automorphism $\tau$ in $G$ and each place $v$ of $k$, this action satisfies the identity
\begin{equation}\label{relunit511}
\tau W_v(l/k) = W_{\eta v}(l/k),
\end{equation}
where $\vphi$ is the canonical homomorphism {\rm (\ref{relunit504})} and $\eta = \vphi(\tau)$.
\end{lemma}

\begin{proof}  Let $u$ be a fixed place of $\Q$.  Then for each place $v$ in $W_u(k/\Q)$ and each place $w$ in $W_v(l/k)$, we have
\begin{equation}\label{relunit514}
[l_w : \Q_u] = [l_w : k_v] [k_v : \Q_u].
\end{equation}
As $l/\Q$ is Galois, the map
\begin{equation*}\label{relunit516}
w \mapsto [l_w : \Q_u]
\end{equation*}
is constant for $w$ in $W_u(l/\Q)$.  Similarly, the extension $k/\Q$ is Galois and therefore the map
\begin{equation*}\label{relunit518}
v \mapsto [k_v : \Q_u]
\end{equation*}
is constant for $v$ in $W_u(k/\Q)$.  It follows from these observations and the identity (\ref{relunit514}), that
\begin{equation}\label{relunit520}
(v, w) \mapsto [l_w : k_v]
\end{equation}
is constant for pairs $(v, w)$ such that $v$ belongs to $W_u(k/\Q)$ and $w$ belongs to $W_v(l/k)$.
Next we recall that the global degree of the extension $l/k$ is the sum of local degrees, so that for
each place $v$ in $W_u(l/k)$ we have
\begin{equation}\label{relunit513}
[l : k] = \sum_{w|v}~[l_w : k_v].
\end{equation}
Then using (\ref{relunit520}) we get the identity
\begin{equation}\label{relunit517}
[l : k] = [l_w : k_v] \bigl|W_v(l/k)\bigr|
\end{equation}
for each pair $(v, w)$ such that $v$ belongs to $W_u(k/\Q)$ and $w$ belongs to $W_v(l/k)$. 
Now (\ref{relunit520}) and (\ref{relunit517}) imply that the map (\ref{relunit507}) is constant for $v$
in $W_u(k/\Q)$.

If $u$ is a place of $\Q$, we have the disjoint union
\begin{equation}\label{relunit525}
W_u(l/\Q) = \bigcup_{v|u} W_v(l/k),
\end{equation}
where the union on the right of (\ref{relunit525}) is over the collection of places $v$ in $W_u(k/\Q)$.  Let $\tau$ belong
to $G$ and let $\vphi(\tau) = \eta$, where $\eta$ is an automorphism in $\Aut(k/\Q)$, and let $w$ be a place
in $W_v(l/k))$.  Then for each point $\beta$ in $k$ we have
\begin{equation*}\label{relunit527}
\|\beta\|_w = \|\beta\|_v,
\end{equation*}
and, as $k/\Q$ is Galois, we have
\begin{equation*}\label{relunit529}
\tau^{-1}(\beta) = \eta^{-1}(\beta)
\end{equation*}
in $k$.  Therefore, we get
\begin{equation}\label{relunit531}
\|\beta\|_{\tau w} = \|\tau^{-1}(\beta)\|_w = \|\eta^{-1}(\beta)\|_v = \|\beta\|_{\eta v}
\end{equation}
at each point $\beta$ in $k$.  The identity (\ref{relunit531}) implies that $\tau w | \eta v$.  As $w$ in $W_v(l/k)$ was arbitrary,
we have
\begin{equation}\label{relunit532}
\tau W_v(l/k) \subseteq W_{\eta v}(l/k).
\end{equation}
It follows from (\ref{relunit507}) that $W_v(l/k)$ and $W_{\eta v}(l/k)$ have the same cardinality.  And it is trivial to check
that $\tau W_v(l/k)$ and $W_v(l/k)$ have the same cardinality.  Hence there is equality in the inclusion (\ref{relunit532}).  
This shows that $G$ acts transitively on the collection of disjoint subsets (\ref{relunit509}), and also establishes the identity 
(\ref{relunit511}).
\end{proof}

We recall from (\ref{unit21}) that
\begin{align}\label{relunit536}
\begin{split}
E_{l/k} &= \big\{\alpha \in F_l : \norm_{l/k}(\alpha) = 1\big\}\\ 
           &= \bigg\{\alpha \in F_l : \text{$\sum_{w|v} \log |\alpha|_w = 0$ for each place $v$ in $W_{\infty}(k/\Q)$}\bigg\}.
\end{split}         
\end{align}
Because $l/\Q$ is Galois, the subextension $l/k$ is Galois, and the group $H$ acts transitively on each subset 
$W_v(l/k)$.  Then it follows from (\ref{exunit0}) and (\ref{relunit536}) that $H$ acts on $E_{l/k}$.  Here we also assume 
that $k/\Q$ is a Galois extension.  We now show that this additional hypothesis implies that the group $G$ acts on $E_{l/k}$.

\begin{lemma}\label{lemfour2}   Assume that the number fields {\rm (\ref{relunit501})} are such that both $l/\Q$ and 
$k/\Q$ are finite, Galois extensions.  Then the group $G = \Aut(l/\Q)$ acts on the subgroup $E_{l/k}$ of relative units.
\end{lemma}

\begin{proof}  As $G$ acts on elements of the group $F_l$, it suffices to show that if $\alpha$ belongs to the subgroup
$E_{l/k} \subseteq F_l$ then $\tau(\alpha)$ belongs to $E_{l/k}$ for each automorphism $\tau$ in $G$.  Therefore 
we suppose that $\tau$ belongs to $G$, and we let $\vphi\bigl(\tau^{-1}\bigr) = \eta$, where $\vphi$ is the canonical 
homomorphism (\ref{relunit504}).  If $\alpha$ belongs to $E_{l/k}$, then using the identity (\ref{relunit511}) in the statement of
Lemma \ref{lemfour1} we find that
\begin{align}\label{relunit538}
\begin{split}
\sum_{w|v} \log |\tau(\alpha)|_w &= \sum_{w|v} \log |\alpha|_{\tau^{-1} w}\\
	&= \sum_{w \in \tau^{-1} W_v(l/k)} \log |\alpha|_w\\
	&= \sum_{w | \eta v} \log |\alpha|_w\\
	&= 0.
\end{split}
\end{align}
We conclude that $\tau(\alpha)$ belongs to $E_{l/k}$.  
\end{proof}

It follows from Lemma \ref{lemfour2} that an element $\alpha \not= 1$ in the group $E_{l/k}$ has an orbit
\begin{equation}\label{orb1}
\{\tau(\alpha) : \tau \in G\} \subseteq E_{l/k}.
\end{equation}
If the subset on the left of (\ref{orb1}) contains $r(l/k)$ multiplicatively independent elements, then we say that 
$\alpha \not= 1$ is a {\it relative Minkowski unit} for the subgroup $E_{l/k}$.  If $k$ is $\Q$, or if $k$ is an imaginary, quadratic extension 
of $\Q$, then $r(k) = 0$, and $E_{l/k} = F_l$.  In this case $\alpha \not= 1$ in $E_{l/k}$ is a relative Minkowski unit for $E_{l/k}$ if 
and only if $\alpha \not= 1$ is a Minkowski unit for $F_l$.  On the other hand, if $l$ is a $\CM$-field, and $k$ is the maximal totally real 
subfield of $l$, then $r(l/k) = r(l) - r(k) = 0$, the subgroup $E_{l/k}$ is trivial, and relative Minkowski units do not exist.  

\section{Relative Units III: A bi-homomorphism}\label{RUIII}

In this section we continue to assume that $l/\Q$ and $k/\Q$ are both Galois extensions, or equivalently that
$H$ is a normal subgroup of $G$.  We define a bi-homomorphism
\begin{equation}\label{epsilon0}
\Delta : F_l \times \Z^N \rightarrow F_l,
\end{equation}
where $N$ is the cardinality of $W_{\infty}(l/\Q)$.  To define the bi-homomorphism (\ref{epsilon0}) we first select an 
archimedean place $\widehat{w}$ in $W_{\infty}(l/\Q)$, and a transversal
\begin{equation}\label{epsilon1}
\Psi = \{\psi_1, \psi_2, \dots , \psi_N\}
\end{equation}
for the left cosets of the stabilizer $G_{\widehat{w}}$.  Thus we have the disjoint union
\begin{equation*}\label{epsilon2}
G = \bigcup_{n = 1}^N \psi_n G_{\widehat{w}}.
\end{equation*}
The bi-homomorphism $\Delta$ depends on the choice of $\widehat{w}$ and on
the choice of the transversal (\ref{epsilon1}), but to simplify notation we suppress this dependence.  We write $\Z^N$
for the free abelian group of rank $N$, and we identify the elements of this group with functions
\begin{equation}\label{epsilon3}
f : G \rightarrow \Z,
\end{equation}
that are constant on left cosets of $G_{\widehat{w}}$.  We recall that $[G : G_{\widehat{w}}] = N$, so that the group of 
such functions does form a free abelian group of rank $N$.  If $G_{\widehat{w}}$ is not trivial, then it has order $2$, and
\begin{equation*}\label{epsilon4}
G_{\widehat{w}} = \{1, \rho\}, \quad\text{where $\rho^2 = 1$.}
\end{equation*}
In this case each left coset of $G_{\widehat{w}}$ has two representatives.  We find that 
\begin{equation*}\label{epsilon5}
\eta G_{\widehat{w}} = \{\eta, \eta \rho\} = \eta \rho \{1 , \rho\} = \eta \rho G_{\widehat{w}}.  
\end{equation*}
Thus a function $f$ as in (\ref{epsilon3}), belongs to $\Z^N$ if and only if it satisfies the identity
\begin{equation}\label{epsilon6}
f(\eta) = f(\eta \rho)
\end{equation}
for each element $\eta$ in $G$.  If $(\alpha, f)$ is an element of the product $F_l \times \Z^N$, we define
\begin{equation*}\label{epsilon9}
\Delta(\alpha, f) = \prod_{n = 1}^N \psi_n(\alpha)^{f(\psi_n)}.
\end{equation*}
As $G$ acts on the group $F_l$, it is obvious that $\Delta(\alpha, f)$ belongs to $F_l$.
If $(\alpha_1, f)$ and $(\alpha_2, f)$ are both elements of $F_l \times \Z^N$, we find that
\begin{align}\label{epsilon13}
\begin{split}
\Delta(\alpha_1, f) \Delta(\alpha_2, f) 
	&= \biggl(\prod_{m = 1}^N \psi_m(\alpha_1)^{f(\psi_m)}\biggr)
		\biggl(\prod_{n = 1}^N \psi_n(\alpha_2)^{f(\psi_n)}\biggr)\\
	&= \prod_{n = 1}^N \psi_n(\alpha_1 \alpha_2)^{f(\psi_n)} = \Delta(\alpha_1 \alpha_2, f).
\end{split}
\end{align}
This shows that for each fixed element $f$ in $\Z^N$, the map
\begin{equation*}\label{epsilon14}
\alpha \mapsto \Delta(\alpha, f)
\end{equation*}
is a homomorphism from $F_l$ into $F_l$.
In a similar manner, if $(\alpha, f_1)$ and $(\alpha, f_2)$ are both elements of $F_l \times \Z^N$, we get
\begin{equation}\label{epsilon17}
\Delta(\alpha, f_1) \Delta(\alpha, f_2) = \Delta(\alpha, f_1 + f_2).
\end{equation}
The identity (\ref{epsilon17}) shows that for each fixed $\alpha$ in $F_l$, the map
\begin{equation}\label{epsilon18}
f \mapsto \Delta(\alpha, f)
\end{equation}
is a homomorphism from $\Z^N$ into $F_l$.  Hence the map $\Delta$ is a bi-homomorphism.  

It will be convenient to use the $\|\ \|_1$-norm on functions $f$ in $\Z^N$.  Therefore we set
\begin{equation*}\label{epsilon19}
\|f\|_1 = \sum_{n = 1}^N \bigl|f(\psi_n)\bigr|
\end{equation*}
for each $f$ in $\Z^N$.  

\begin{lemma}\label{lemdelta1}  If $(\alpha, f)$ is an element of $F_l \times \Z^N$, then
\begin{equation}\label{epsilon20}
h\bigl(\Delta(\alpha, f)\bigr) \le \|f\|_1 h(\alpha).
\end{equation}
\end{lemma}

\begin{proof}  The map 
\begin{equation*}\label{epsilon21}
n \mapsto h\bigl(\psi_n(\alpha)\bigr)
\end{equation*}
is constant.  Thus we have
\begin{align*}\label{epsilon22}
\begin{split}
2 h\bigl(\Delta(\alpha, f)\bigr) &= \sum_{w | \infty} \bigl|\log |\Delta(\alpha, f)|_w\bigr|\\
	&= \sum_{w | \infty} \biggl|\sum_{n = 1}^N f(\psi_n) \log |\psi_n(\alpha)|_w\biggr|\\
	&\le \sum_{n = 1}^N \bigl|f(\psi_n)\bigr| \sum_{w | \infty} \bigl|\log |\psi_n(\alpha)|_w\bigr|\\
	&= 2 \sum_{n = 1}^N \bigl|f(\psi_n)\bigr| h\bigl(\psi_n(\alpha)\bigr)\\
	&= 2 \|f\|_1 h(\alpha).
\end{split}
\end{align*}
This proves the lemma.
\end{proof}

We have noted that the group $G$ acts on $F_l$.  The group $G$ also acts on $\Z^N$.  More precisely, if $\eta$ 
belongs to $G$ and $x \mapsto f(x)$ belongs to $\Z^N$, we denote the action of $\eta$ on $f$ by $[\eta, f]$, where
\begin{equation}\label{eta1}
[\eta, f](x) = f\bigl(\eta^{-1} x\bigr).
\end{equation}
If $f$ satisfies
\begin{equation*}\label{eta2}
f(\tau) = f(\tau \rho)
\end{equation*}
for each $\tau$ in $G$, then it is obvious that
\begin{equation}\label{eta3}
f\bigl(\eta^{-1} \tau\bigr) = f\bigl(\eta^{-1} \tau \rho\bigr)
\end{equation}
for each $\tau$ in $G$.  That is, if $\eta$ belongs to $G$ and $f$ belongs to $\Z^N$, then $[\eta, f]$ belongs to $\Z^N$. 
It is clear that
\begin{equation}\label{eta5}
[1, f](x) = f(x).
\end{equation}
If both $\eta_1$ and $\eta_2$ belong to $G$, then
\begin{align}\label{eta6}
\begin{split}
\bigl[\eta_1, [\eta_2, f]\bigr] (x) & = [ \eta_{2} , f ] (\eta_{1}^{-1} x)\\
                                             &= f\bigl(\eta_2^{-1} \eta_1^{-1} x\bigr)\\
                                             &= f\bigl((\eta_1 \eta_2)^{-1} x\bigr)\\
                                             &= \bigl[\eta_1 \eta_2, f\bigr] (x).
\end{split}
\end{align}
The identities (\ref{eta3}), (\ref{eta5}), and (\ref{eta6}), verify that (\ref{eta1}) defines an action of the group $G$ on
the collection of functions in $\Z^N$.

The action of $G$ on the $\Z^N$ occurs in the following identities.

\begin{lemma}\label{lemdelta2}  Let $\eta$ belong to $G$, and let $(\alpha, f)$ be a point in $F_l \times \Z^N$.
If $l/\Q$ is a totally real Galois extension, then $G_{\widehat{w}}$ is trivial and
\begin{equation}\label{eta44}
\eta\bigl(\Delta(\alpha, f)\bigr) = \Delta\bigl(\alpha, [\eta, f]\bigr).
\end{equation}
If $l/\Q$ is a totally complex Galois extension, then $G_{\widehat{w}}$ is cyclic of order $2$,
\begin{equation}\label{eta49}
G_{\widehat{w}} = \{1, \rho\},\quad\text{where $\rho^2 = 1$},
\end{equation}
and
\begin{equation}\label{eta53}
\eta\bigl(\Delta(\alpha \rho(\alpha), f)\bigr) = \Delta\bigl(\alpha \rho(\alpha), [\eta, f]\bigr).
\end{equation}
\end{lemma}

\begin{proof}  We assume that $l/\Q$ totally real.  Then we have
\begin{equation*}\label{eta57}
G = \{\psi_1, \psi_2, \dots , \psi_N\},
\end{equation*}
and
\begin{align}\label{eta61}
\begin{split}
\eta\bigl(\Delta(\alpha, f)\bigr) &= \prod_{n = 1}^N \bigl(\eta \psi_n (\alpha)\bigr)^{f(\psi_n)}\\
		&= \prod_{n = 1}^N \psi_n(\alpha)^{f(\eta^{-1} \psi_n)}\\
		&= \Delta(\alpha, [\eta, f]).
\end{split}
\end{align}

Now assume that $l/\Q$ is totally complex.  Then
\begin{equation*}\label{eta65}
G = \{\psi_1, \psi_2, \dots , \psi_N\} \cup \{\psi_1\rho, \psi_2 \rho, \dots , \psi_N \rho\} 
		= \Psi \cup \Psi \rho,
\end{equation*}
and
\begin{align}\label{eta68}
\begin{split}
\eta\bigl(\Delta(\alpha, f) \Delta(\rho(\alpha), f)\bigr) 
	&= \prod_{m = 1}^N\bigl(\eta \psi_m(\alpha)\bigr)^{f(\psi_m)} 
			\prod_{n = 1}^N \bigl(\eta \psi_n \rho (\alpha)\bigr)^{f(\psi_n \rho)}\\
	&= \prod_{m = 1}^N\bigl(\psi_m(\alpha)\bigr)^{f(\eta^{-1}\psi_m)} 
			\prod_{n = 1}^N \bigl(\psi_n \rho (\alpha)\bigr)^{f(\eta^{-1}\psi_n \rho)}\\
	&= \Delta(\alpha, [\eta, f]) \Delta(\rho(\alpha),[\eta, f]).
\end{split}
\end{align}
When (\ref{epsilon13}) and (\ref{eta68}) are combined, we obtain the identity (\ref{eta53}) for each automorphism $\eta$ in $G$.
\end{proof}

As $F_l$ is a free abelian group of rank $N - 1$, it follows that for each
$\alpha$ in $F_l$ the kernel of the homomorphism (\ref{epsilon18}) has rank greater than or equal to $1$.  The following result
is an immediate consequence of Theorem \ref{thmspecial1}

\begin{theorem}\label{thmdelta1}  Let $\beta$ in $F_l$ be a special Minkowski unit with respect to $\widehat{w}$.  Then the image
\begin{equation*}\label{epsilon23}
\big\{\Delta(\beta, f) : f \in \Z^N\big\} 
\end{equation*}
of the homomorphism {\rm (\ref{epsilon18})} is a subgroup of $F_l$ with rank $N - 1$.
\end{theorem}

\begin{proof}  It suffices to show that the kernel
\begin{align}\label{epsilon25}
\begin{split}
\K(\beta) &= \big\{f \in \Z^N : \Delta(\beta, f) = 1\big\}\\
	       &= \big\{f \in \Z^N : \text{$\sum_{n = 1}^N f(\psi_n) \log |\psi_n(\beta)|_w = 0$
	       				for each $w$ in $W_{\infty}(l/\Q)$}\big\}\\
	       &= \big\{f \in \Z^N : \text{$\sum_{n = 1}^N f(\psi_n) \log |\psi_m^{-1}\psi_n(\beta)|_{\widehat{w}} = 0$
	       				for $m = 1, 2, \dots , N$}\big\}\\
\end{split}
\end{align}
has rank $1$.  If we write $f$ in $\Z^N$ as a (column) vector $\bff$, then $\K(\beta)$ is the null space of the linear transformation  
\begin{equation*}\label{epsilon27}
\bff \mapsto M(\beta, \Psi, \widehat{w}) \bff,
\end{equation*}
where $M(\beta, \Psi, \widehat{w})$ is the $N \times N$ matrix defined in (\ref{new112}).  As the matrix 
$M(\beta, \Psi, \widehat{w})$ has rank $N-1$, it follows that the kernel $\K(\beta)$ has rank $1$.
\end{proof}

If $f$ belongs to the kernel $\K(\beta)$ defined in (\ref{epsilon25}) and $f$ is not identically zero, then it follows from (v)
in the statement of Theorem \ref{thmspecial1}, that the co-ordinate function 
\begin{equation*}\label{epsilon31}
n \mapsto f(\psi_n)
\end{equation*}
takes only positive values, or it takes only negative values.  This will play a crucial role in our construction of relative Minkowski units.

We recall that $H$ acts transitively on each collection of places $W_v(l/k)$, where $v$ is a place in $W_{\infty}(k/\Q)$, and 
$G$ acts transitively on the collection of places $W_{\infty}(l/\Q)$.  Also by Lemma \ref{lemfour1}, the group $G$ acts 
transitively on the collection of subsets
\begin{equation}\label{delta2}
\big\{W_v(l/k) : v \in W_u(k/\Q)\big\}.
\end{equation}
If $\tau$ is an automorphism in $G$, and $\vphi : G \rightarrow G/H$ is the canonical homomorphism, then at each place $v$ of $k$, 
the action of $\tau$ on the subsets in the collection (\ref{delta2}) is given by
\begin{equation}\label{delta4}
\tau W_v(l/k) = W_{\eta v}(l/k),
\end{equation}
where $\eta = \vphi(\tau)$.  If $\tau_1$ and $\tau_2$ belong to $G$, then it follows from (\ref{delta4}) that 
\begin{equation}\label{delta8}
\tau_1 W_v(l/k) = \tau_2 W_v(l/k)
\end{equation}
if and only if $\vphi(\tau_1) = \vphi(\tau_2)$.  That is, (\ref{delta8}) holds if and only if $\tau_1 H = \tau_2 H$. 

There is a further implication of (\ref{delta4}) that will be useful.  Let $\alpha$ belong to $F_l$, and let $v$ be a place in 
$W_{\infty}(k/\Q)$.  Then for each $\tau$ in $G$ we have
\begin{align*}\label{delta13}
\begin{split}
\sum_{w | v} \log \bigl|\tau^{-1}(\alpha)\bigr|_w &= \sum_{w \in W_v(l/k)} \log |\alpha|_{\tau w} \\
	&= \sum_{w \in \tau W_v(l/k)} \log |\alpha|_w \\
	&= \sum_{w \in W_{\eta v}(l/k)} \log |\alpha|_w,
\end{split}
\end{align*}
where $\eta = \vphi(\tau)$.  It follows that the map
\begin{equation}\label{delta15}
\tau \mapsto \sum_{w | v} \log \bigl|\tau^{-1}(\alpha)\bigr|_w
\end{equation}
from $G$ into $\R$, depends only on $\eta = \vphi(\tau)$, and therefore (\ref{delta15}) is constant for $\tau$ 
restricted to a coset of $H$.  This observation was already used in the identity (\ref{relunit538}). 

We select a place $\widehat{w}$ in $W_{\infty}(l/\Q)$.  As before we write
\begin{equation*}\label{delta19}
G_{\widehat{w}} = \{\tau \in G : \tau \widehat{w} = \widehat{w}\},
\end{equation*}
for the stabilizer of $\widehat{w}$ in $G$.  We will continue to write
\begin{equation*}\label{delta29}
|G| = [l : \Q] = \begin{cases}  N&  \text{if $l/\Q$ is a totally real Galois extension,}\\
 					    2N&  \text{if $l/\Q$ is a totally complex Galois extension,}\end{cases}
\end{equation*}
so that
\begin{equation*}\label{delta33}
[G : G_{\widehat{w}}] = \bigl|W_{\infty}(l/\Q)\bigr| = N.
\end{equation*}
Let $I = [G : G_{\widehat{w}} H]$, and let 
\begin{equation}\label{delta37}
\{\tau_1, \tau_2, \dots , \tau_I\} \subseteq G
\end{equation}
be a left transversal for the subgroup $G_{\widehat{w}} H$ in $G$.  Then we have the disjoint union
\begin{equation}\label{delta41}
G = \bigcup_{i = 1}^I \tau_i G_{\widehat{w}} H.
\end{equation}
Using Lemma \ref{lemfour1}, the rank of the group $E_{l/k}$ of relative units is given by
\begin{align*}\label{delta45}
\begin{split}
r(l/k) &= r(l)  - r(k)\\
        &= \bigl(\bigl|W_{\infty}(l/\Q)\bigr| - 1\bigr) - \bigl(\bigl|W_{\infty}(k/\Q)\bigr| - 1\bigr)\\
        &= [G : G_{\widehat{w}}] - [G : G_{\widehat{w}} H]\\
          &= N - I.
\end{split}        
\end{align*}
Similarly, let $J = [G_{\widehat{w}} H : G_{\widehat{w}}]$, and let
\begin{equation}\label{delta53}
\{\sigma_1, \sigma_2, \dots , \sigma_J\} \subseteq G_{\widehat{w}} H
\end{equation}  
be a left transversal for the subgroup $G_{\widehat{w}}$ in $G_{\widehat{w}} H$.  Then we have the disjoint union
\begin{equation}\label{delta55}
G_{\widehat{w}} H = \bigcup_{j = 1}^J \sigma_j G_{\widehat{w}}.
\end{equation}
If $k/\Q$ is totally real then $G_{\widehat{w}} H = H$, and it is obvious that
\begin{equation}\label{delta56}
\{\sigma_1, \sigma_2, \dots , \sigma_J\} = H.
\end{equation}
If $k/\Q$ is totally complex then $G_{\widehat{w}} = \{1, \rho\}$ has order $2$, and $G_{\widehat{w}} \cap H$ is trivial.  It follows that
\begin{equation*}\label{delta57}
G_{\widehat{w}} H = H \cup \rho H = H \cup H \rho,\quad\text{and}\quad [G_{\widehat{w}} H : G_{\widehat{w}}] = |H|.
\end{equation*}
In this case we select the transversal (\ref{delta53}) so that
\begin{equation}\label{delta58}
\{\sigma_1, \sigma_2, \dots , \sigma_J\} = H.
\end{equation}

Combining (\ref{delta41}) and (\ref{delta55}), we find that
\begin{equation}\label{delta63}
G = \bigcup_{i = 1}^I \tau_i G_{\widehat{w}} H
    = \bigcup_{i = 1}^I \tau_i \biggl(\bigcup_{j = 1}^J \sigma_j G_{\widehat{w}}\biggr)
    = \bigcup_{i = 1}^I \bigcup_{j = 1}^J \Bigl(\tau_i \sigma_j G_{\widehat{w}}\Bigr).
\end{equation}
It follows that
\begin{equation*}\label{delta65}
\big\{\tau_i \sigma_j : \text{$i = 1, 2, \dots , I$, and $j = 1, 2, \dots , J$}\big\}
\end{equation*}
is a transversal for the subgroup $G_{\widehat{w}}$ in $G$.  We also have
\begin{equation*}\label{delta67}
N = [G : G_{\widehat{w}}] = [G : G_{\widehat{w}} H] [G_{\widehat{w}} H : G_{\widehat{w}}] = I J.
\end{equation*}

Next we require a variant of the fact that (\ref{delta15}) depends only on $\eta = \vphi(\tau)$.

\begin{lemma}\label{lemdelta3}  Let $\tau_i$ be a coset representative in {\rm (\ref{delta37})}, and let
$\sigma_j$ be a coset representative in {\rm (\ref{delta53})}.  Let $\alpha$ be an element of the group $F_l$, and
let $v$ be a place in $W_{\infty}(k/\Q)$.  Then the map
\begin{equation}\label{delta69}
(\tau_i, \sigma_j) \mapsto \sum_{w | v} \log\bigl|\tau_i \sigma_j \alpha\bigr|_w
\end{equation}
depends only on $\tau_i$, and not on $\sigma_j$. 
\begin{equation*}\label{delta71}
\end{equation*}
\end{lemma}

\begin{proof}   Let $\vphi : G \rightarrow G/H$ be the canonical homomorphism.  Because $\Aut(k/\Q)$ acts transitively on 
$W_{\infty}(k/\Q)$, and $\Aut(k/\Q)$ is isomorphic to $G/H$, there exists $\widehat{v}$ in $W_{\infty}(k/\Q)$ so that
\begin{equation*}\label{delta73}
\vphi(\tau_i) \widehat{v} = v,\quad\text{and}\quad \tau_i W_{\widehat{v}}(l/k) = W_v(l/k).
\end{equation*}
Using the identity (\ref{relunit511}) in the statement of Lemma \ref{lemfour1}, we find that
\begin{align}\label{delta77}
\begin{split}
\sum_{w | v} \log\bigl|\tau_i \sigma_j \alpha\bigr|_w &= \sum_{w \in W_v(l/k)} \log\bigl|\sigma_j \alpha \bigr|_{\tau_i^{-1} w}\\
	&= \sum_{w \in \tau_i W_{\widehat{v}}(l/k)} \log\bigl|\sigma_j \alpha \bigr|_{\tau_i^{-1} w}\\
	&= \sum_{w | \widehat{v}} \log\bigl|\sigma_j \alpha \bigr|_w.
\end{split}
\end{align}
Because $\sigma_j$ belongs to $H$, its image $\vphi(\sigma_j)$ is trivial.  In this case (\ref{relunit511}) implies that
\begin{equation}\label{delta83}
\sum_{w | \widehat{v}} \log\bigl|\sigma_j \alpha \bigr|_w = \sum_{w | \widehat{v}} \log\bigl|\alpha \bigr|_w.
\end{equation}
Now (\ref{delta77}) and (\ref{delta83}) show that the map (\ref{delta69}) depends on $\tau_i$, but not on $\sigma_j$.
\end{proof}

Let $\eL_{l/k} \subseteq \Z^N$ be the subgroup
\begin{equation*}\label{epsilon40}
\eL_{l/k} = \{f \in \Z^N : \text{$\sum_{j = 1}^J f(\tau_i \sigma_j) = 0$ for each $i = 1, 2, \dots , I$}\}. 
\end{equation*}
The subsets
\begin{equation*}\label{epsilon43}
\tau_i G_{\widehat{w}} H = \bigcup_{j = 1}^J \tau_i \sigma_j G_{\widehat{w}},
\end{equation*}
where $i = 1, 2, \dots , I$, are the distinct cosets of $G_{\widehat{w}} H$ in $G$ and are therefore disjoint.  Hence
the linear conditions defining the subgroup $\eL_{l/k}$ are independent.  It follows that 
\begin{equation*}\label{epsilon45}
\rank \eL_{l/k} = N - I = r(l/k).
\end{equation*}

\begin{theorem}\label{thmdelta2}  Let $\beta$ in $F_l$ be a special Minkowski unit.  Then the image 
\begin{equation}\label{epsilon49}
\big\{\Delta(\beta, f) : f \in \eL_{l/k}\big\}
\end{equation}
of the homomorphism {\rm (\ref{epsilon18})} restricted to $\eL_{l/k}$, is a subgroup of $E_{l/k}$ with rank $r(l/k)$.
\end{theorem}

\begin{proof}  Let $v$ be a place in $W_{\infty}(k/\Q)$, and let $f$ belong to $\eL_{l/k}$.  Then we have
\begin{align}\label{epsilon60}
\begin{split}
\sum_{w | v} \log \bigl|\Delta(\beta, f)\bigr|_w 
  &= \sum_{w | v} \biggl(\sum_{i = 1}^I \sum_{j = 1}^J f(\tau_i \sigma_j) \log |\tau_i \sigma_j \beta|_w\biggr)\\
  &= \sum_{i = 1}^I \sum_{j = 1}^J f(\tau_i \sigma_j) \biggl(\sum_{w | v} \log |\tau_i \sigma_j \beta|_w\biggr).
\end{split}
\end{align}
By Lemma \ref{lemdelta3} the sum
\begin{equation*}\label{epsilon65}
\sum_{w | v} \log |\tau_i \sigma_j \beta|_w = c(\tau_i)
\end{equation*}
depends on $\tau_i$, but not on $\sigma_j$.  Hence on the right hand side of (\ref{epsilon60}) we get
\begin{equation}\label{epsilon69}
\sum_{j = 1}^J f(\tau_i \sigma_j) \biggl(\sum_{w | v} \log |\tau_i \sigma_j \beta|_w\biggr)
	= c(\tau_i) \sum_{j = 1}^J f(\tau_i \sigma_j) = 0
\end{equation}
for each $i = 1, 2, \dots , I$.  Combining (\ref{epsilon60}) and (\ref{epsilon69}), we conclude that
\begin{equation*}\label{epsilon73}
\sum_{w | v} \log \bigl|\Delta(\beta, f)\bigr|_w = 0
\end{equation*}
for each place $v$ in $W_{\infty}(k/\Q)$.  We have shown that if $f$ belongs to the subgroup $\eL_{l/k}$, then $\Delta(\beta, f)$
belongs to the subgroup $E_{l/k}$.

Next we prove that the image (\ref{epsilon49}) has rank $r(l/k)$.  As $\eL_{l/k}$ has rank $r(l/k)$, it suffices to show
that the map $f \mapsto \Delta(\beta, f)$, restricted to the subgroup $\eL_{l/k}$, is injective.  This will follow if we show that
its kernel is trivial.  That is, it suffices to show that
\begin{equation}\label{epsilon76}
\K(\beta) \cap \eL_{l/k} = \{\bo\},
\end{equation}
where $\K(\beta)$ is defined by (\ref{epsilon25}).   We have already noted that if $f \not= \bo$ belongs to $\K(\beta)$, then
\begin{equation*}\label{epsilon81}
n \mapsto f(\psi_n)
\end{equation*}
takes only positive values, or it takes only negative values.  However, if $f\not= \bo$ belongs to $\eL_{l/k}$, then
\begin{equation*}\label{epsilon85}
\sum_{j = 1}^J f(\tau_i \sigma_j) = 0
\end{equation*}
for each $i = 1, 2, \dots , I$.  Therefore the only point in the intersection (\ref{epsilon76}) is $\bo$, and the theorem is proved.
\end{proof}

As in the proof of Theorem \ref{thmdelta2}, the map $f \mapsto \Delta(\beta, f)$ restricted to $\eL_{l/k}$ is injective. Therefore 
we get the following result.

\begin{corollary}\label{cordelta1}  Let $\beta$ in $F_l$ be a special Minkowski unit, and let 
\begin{equation*}\label{epsilon87}
\{f_1, f_2, \dots f_R\},\quad\text{where $R = r(l/k)$,} 
\end{equation*}
be linearly independent elements in the free group $\eL_{l/k}$.  Then the elements of the set
\begin{equation*}\label{epsilon91}
\big\{\Delta(\beta, f_1), \Delta(\beta, f_2), \dots , \Delta(\beta, f_R)\big\}
\end{equation*}
are multiplicatively independent relative units in $E_{l/k}$.
\end{corollary}

\section{Relative Units IV: Existence}\label{E}

We continue to assume that $l/\Q$ and $k/\Q$ are finite, Galois extensions such that
\begin{equation*}\label{mink0}
\Q \subseteq k \subseteq l,\quad\text{and}\quad 1 \le r(k) < r(l),
\end{equation*}
and we let $\widehat{w}$ denote a particular archimedean place of $l$.  If $l/\Q$ is a totally complex Galois
extension, then $G_{\widehat{w}}$ is cyclic of order $2$, and we write
\begin{equation}\label{mink1}
G_{\widehat{w}} = \{1, \rho\},\quad\text{where $\rho^2 = 1$}.
\end{equation}
In this section we apply Lemma \ref{lemmap3} with
\begin{equation*}\label{mink3}
K = G_{\widehat{w}},\quad H = \Aut(l/k),\quad\text{and}\quad G = \Aut(l/\Q),
\end{equation*}
and with $\lambda$ in $\eL_{l/k}$ defined by (\ref{map49}).  Let $\{\tau_1, \tau_2, \dots , \tau_I\}$ be a transversal for the left 
cosets of $G_{\widehat{w}} H$ in $G$, and let $\{\sigma_1, \sigma_2, \dots , \sigma_J\}$ be a transversal for the left cosets 
of $G_{\widehat{w}}$ in $G_{\widehat{w}} H$.  For each $i = 1, 2, \dots , I$, let 
\begin{equation*}\label{mink4}
\J_i \subseteq \{1, 2, \dots , J\}
\end{equation*}
be a subset of cardinality $|\J_i| = J - 1$.  Using (\ref{delta56}) and (\ref{delta58}) we have
\begin{equation*}\label{mink5}
J = [G_{\widehat{w}} H : G_{\widehat{w}}] = |H| = [l : k].
\end{equation*}
Then letting $\tau_1$ be a coset representative in $G_{\widehat{w}} H$, and letting $\sigma_1$ be a coset representative
in $G_{\widehat{w}}$, we get
\begin{equation}\label{mink6}
\begin{split}
\|\lambda\|_1 &= \sum_{i = 1}^I \sum_{j = 1}^J |\lambda(\tau_i \sigma_j)| = \sum_{j = 1}^J |\lambda(\tau_1 \sigma_j)|\\
		     &= |\lambda(\tau_1 \sigma_1)| + \sum_{j = 2}^J |\lambda(\tau_1 \sigma_j)| = (J - 1) + (J - 1)\\
		     &= 2\bigl([l : k] - 1\bigr).
\end{split}
\end{equation}

The following result establishes the existence of relative Minkowski units.

\begin{theorem}\label{thmmink1}  Let $\beta$ in $F_l$ be a special Minkowski unit with respect to the infinite place $\widehat{w}$.
\begin{itemize}
\item[(i)]  If $l/\Q$ is a totally real Galois extension, then the elements in the set
\begin{equation}\label{mink7}
\big\{\tau_i \sigma_j \bigl(\Delta(\beta, \lambda)\bigr) : \text{$i = 1, 2, \dots , I$ and $j \in \J_i$}\big\}
\end{equation}
are multiplicatively independent relative units in $E_{l/k}$, and satisfy
\begin{equation}\label{mink9}
h\bigl(\Delta(\beta, \lambda)\bigr) \le 2([l : k] - 1) h(\beta).
\end{equation}
\item[(ii)]  If $l/\Q$ is a totally complex Galois extension, then the elements in the set
\begin{equation}\label{mink13}
\big\{\tau_i \sigma_j \bigl(\Delta(\beta \rho(\beta)\bigr), \lambda)\bigr) : \text{$i = 1, 2, \dots , I$ and $j \in \J_i$}\big\}
\end{equation}
are multiplicatively independent relative units in $E_{l/k}$, and satisfy
\begin{equation}\label{mink17}
h\bigl(\Delta(\beta \rho(\beta), \lambda)\bigr) \le 4([l : k] - 1) h(\beta).
\end{equation}
\end{itemize}
\end{theorem}

\begin{proof}  By Lemma \ref{lemmap3} the functions
\begin{equation*}\label{mink25}
\big\{[\tau_i \sigma_j, \lambda] : \text{$i = 1, 2, \dots , I$ and $j \in \J_i$}\big\} \subseteq \eL_{l/k}
\end{equation*}
are linearly independent.  Let $\beta$ in $F_l$ be a special Minkowski unit.  Corollary \ref{cordelta1} implies that 
the algebraic numbers
\begin{equation}\label{mink30}
\big\{\Delta\bigl(\beta, [\tau_i \sigma_j, \lambda]\bigr) : \text{$i = 1, 2, \dots , I$ and $j \in \J_i$}\big\}
\end{equation}
are multiplicatively independent relative units in $E_{l/k}$.  If $l/\Q$ is a totally complex Galois extension such that
\begin{equation}\label{mink31}
G_{\widehat{w}} = \{1, \rho\},\quad\text{where $\rho^2 = 1$},
\end{equation}
then it follows from Lemma \ref{lemspecial3} that $\beta \rho(\beta)$ is a special Minkowski unit with respect to $\widehat{w}$.
In this case Corollary \ref{cordelta1} asserts that the algebraic numbers
\begin{equation}\label{mink32}
\big\{\Delta\bigl(\beta \rho(\beta), [\tau_i \sigma_j, \lambda]\bigr) : \text{$i = 1, 2, \dots , I$ and $j \in \J_i$}\big\}
\end{equation}
are multiplicatively independent relative units in $E_{l/k}$.

Assume that $l/\Q$ is a totally real Galois extension.  Then it follows from (\ref{eta44}) in the statement of 
Lemma \ref{lemdelta2}, that 
\begin{equation}\label{mink35}
\tau_i \sigma_j\bigl(\Delta(\beta, \lambda)\bigr) = \Delta\bigl(\beta, [\tau_i \sigma_j, \lambda]\bigr)
\end{equation}
for each $i = 1, 2, \dots , I$ and each $j \in \J_i$.  Combining (\ref{mink30}) and (\ref{mink35}), we find that the conjugate 
algebraic numbers in the set
\begin{equation*}\label{mink40}
\big\{\tau_i \sigma_j \bigl(\Delta\bigl(\beta, \lambda\bigr)\bigr) : \text{$i = 1, 2, \dots , I$ and $j \in \J_i$}\big\}
\end{equation*}
are multiplicatively independent relative units in $E_{l/k}$.  That is, $\Delta\bigl(\beta, \lambda\bigr)$ is a relative Minkowski
unit in $E_{l/k}$.  Applying the inequality (\ref{epsilon20}) and (\ref{mink6}), we get the bound
\begin{equation*}\label{mink45}
h\bigl(\Delta\bigl(\beta, \lambda\bigr)\bigr) \le \|\lambda\|_1 h(\beta) = 2([l : k] - 1) h(\beta).
\end{equation*}
This verifies the inequality (\ref{mink9}).

Now assume that $l/\Q$ is a totally complex Galois extension, and let $G_{\widehat{w}}$ be given by (\ref{mink31}).
It follows from (\ref{eta53}) in the statement of Lemma \ref{lemdelta2}, that
\begin{equation}\label{mink55}
\tau_i \sigma_j\bigl(\Delta(\beta \rho(\beta), \lambda)\bigr) = \Delta\bigl(\beta \rho(\beta), [\tau_i \sigma_j, \lambda]\bigr)
\end{equation}
for each $i = 1, 2, \dots , I$ and each $j \in \J_i$.  In this case we combine (\ref{mink30}) and (\ref{mink55}) and
conclude that the conjugate algebraic numbers in the set
\begin{equation*}\label{mink63}
\big\{\tau_i \sigma_j \bigl(\Delta(\beta \rho(\beta)\bigr), \lambda)\bigr) : \text{$i = 1, 2, \dots , I$ and $j \in \J_i$}\big\}
\end{equation*}
are multiplicatively independent relative units in $E_{l/k}$.  Therefore $\Delta(\beta \rho(\beta), \lambda)$ is a 
relative Minkowski unit in $E_{l/k}$.  It follows from the bound (\ref{epsilon20}) and (\ref{mink6}), that
\begin{align*}\label{mink69}
h\bigl(\Delta(\beta \rho(\beta), \lambda)\bigr) &\le \|\lambda\|_1 h(\beta \rho(\beta))\\ 
	&\le 2([l : k] - 1) \bigl(h(\beta) + h(\rho(\beta))\bigr)\\
	&= 4([l : k] - 1) h(\beta).
\end{align*} 
This establishes the inequality (\ref{mink17}).
\end{proof}

We now prove Theorem \ref{maintheorem2}.  Let $l/\Q$ be a Galois extension with $N$ archimedean places, and let $\widehat{w}$ 
be a particular archimedean place of $l$.  Let $\eta_1, \eta_2, \dots , \eta_{r(l)}$ be a basis for the group $F_l$, where $r(l) = N - 1$.  
By Theorem \ref{thmspecial1} there exists a special Minkowski unit $\beta$ with respect to $\widehat{w}$, such that
\begin{equation}\label{mink71}
h(\beta) \le 2 \sum_{j=1}^{r(l)} h(\eta_j).
\end{equation}

If $l/\Q$ is totally real, then it follows from (i) in the statement of Theorem \ref{thmmink1} that 
\begin{equation*}\label{mink77}
\gamma = \Delta(\beta, \lambda)
\end{equation*}
is a relative Minkowski unit for the subgroup $E_{l/k}$.  Combining the inequalities (\ref{mink9}) and (\ref{mink71}), we find that
\begin{equation}\label{mink81}
h (\gamma) = h\bigl(\Delta(\beta, \lambda)\bigr) \le 2\bigl([l : k] - 1\bigr) h(\beta) \le 4\bigl([l : k] - 1\bigr) \sum_{j = 1}^{r(l)} h(\eta_j),
\end{equation}
and this verifies (\ref{first52}).

Now suppose that $l/\Q$ is a totally complex Galois extension.  In this case the stabilizer $G_{\widehat{w}}$ is cyclic of order 
$2$, and we use (\ref{mink1}).  It follows from (ii) in the statement of Theorem \ref{thmmink1} that
\begin{equation*}\label{mink95}
\gamma = \Delta(\beta \rho(\beta), \lambda)
\end{equation*}
is a relative Minkowski unit for the subgroup $E_{l/k}$.  To complete the proof we combine the inequalities (\ref{mink17})
and (\ref{mink71}).  We find that
\begin{equation*}\label{mink94}
h (\gamma) = h\bigl(\Delta(\beta \rho(\beta), \lambda)\bigr) \le 4\bigl([l : k] - 1\bigr) h(\beta) 
	           \le 8\bigl([l : k] - 1\bigr) \sum_{j = 1}^{r(l)} h(\eta_j),
\end{equation*}
and this proves (\ref{first55}).

\section*{Acknowledgments}

We are grateful to the anonymous referees for their helpful comments and suggestions.


\end{document}